\documentclass[reqno,tbtags,a4paper,12pt]{amsart}

\usepackage[in]{fullpage}
\usepackage[matrix,arrow,curve]{xy}
\usepackage{tikz}
\usepackage{tikz-cd}
\usetikzlibrary{arrows, matrix, bbox}
\usetikzlibrary{snakes}

\usepackage{amsmath,amssymb,amsthm,amsfonts,amscd}
\usepackage{graphics,graphicx}

\definecolor{mygreen}{rgb}{0, 0.4, 0}
\definecolor{myblue}{rgb}{0, 0, 0.5}

\newcommand{\uE}{\underline{E}}
\newcommand{\uF}{\underline{\mathcal{F}}}
\newcommand{\CA}{\mathcal{A}}
\newcommand{\Ind}{\mathop{\mathrm{Ind}}\nolimits}
\newcommand{\cd}{\mathop{\mathrm{cd}}\nolimits}
\newcommand{\CU}{\mathbf{U}}

\newcommand{\BP}{\mathrm{BP}}
\newcommand{\bp}{\mathrm{bp}}
\newcommand{\Mod}{\mathop{\mathrm{Mod}}\nolimits}
\newcommand{\PMod}{\mathop{\mathrm{PMod}}\nolimits}
\newcommand{\bk}{\mathbf{k}}

\newcommand{\Homeo}{\mathrm{Homeo}}
\newcommand{\Sp}{\mathrm{Sp}}
\newcommand{\cur}{\mathcal{C}}
\newcommand{\fB}{\mathfrak{B}}

\newcommand{\CE}{\mathcal{E}}
\newcommand{\CF}{\mathcal{F}}
\newcommand{\CS}{\mathcal{U}}
\newcommand{\Stab}{\mathop{\mathrm{Stab}}\nolimits}
\newcommand{\I}{\mathcal{I}}
\newcommand{\J}{\mathcal{J}}
\newcommand{\T}{\mathcal{T}}
\newcommand{\h}{\mathfrak{h}}
\newcommand{\B}{\mathcal{B}}
 \newcommand{\M}{\mathcal{M}}
 \newcommand{\Q}{\mathbb{Q}}
\newcommand{\F}{\mathbb{F}}
\newcommand{\Z}{\mathbb{Z}}
\newcommand{\R}{\mathbb{R}}
\newcommand{\bb}{\mathbf{b}}
\newcommand{\rank}{\mathop{\mathrm{rank}}}
\newcommand{\hE}{\widehat{E}}
\newcommand{\hPi}{\widehat{\Pi}}
\newcommand{\fN}{\mathfrak{N}}
\newcommand{\HH}{\mathrm{H}}

\newcommand{\bgamma}{\boldsymbol{\gamma}}

\newtheorem{theorem}{Theorem}[section] 
\newtheorem{propos}[theorem] {Proposition}
\newtheorem{cor}[theorem] {Corollary}
\newtheorem{lem}[theorem]{Lemma}
\newtheorem{fact}[theorem] {Fact}
\theoremstyle{definition}
\newtheorem{remark}[theorem]{Remark}
\newtheorem{definition}[theorem]{Definition}

\numberwithin{equation}{section}

\author{Alexander A. Gaifullin}

\thanks{The work is supported by a grant of the President of the Russian Federation (grant MD-2907.2017.1), by the Russian Foundation for Basic Research (grant 18-51-50005), and by the Theoretical Physics and Mathematics Advancement Foundation ``BASIS'' (grant 22-7-2-10-1).}

\subjclass[2020]{57K20 (Primary); 57M07, 20J05, 20F34 (Secondary)}

\title{On infinitely generated homology of Torelli groups}
\date{}

\address{Steklov Mathematical Institute of Russian Academy of Sciences, Moscow, Russia}

\address{Skolkovo Institute of Science and Technology, Moscow, Russia}

\address{Lomonosov Moscow State University, Russia}

\address{Institute for Information Transmission Problems of the Russian Academy of Sciences (Kharkevich Institute), Moscow, Russia}

\email{agaif@mi-ras.ru}

\keywords{Torelli group, homology of groups, complex of cycles, abelian cycle, spectral sequence}

\begin{document}

\begin{abstract}
Let $\I_g$ be the Torelli group of an oriented closed surface~$S_g$ of genus~$g$, that is, the kernel of the action of the mapping class group on the first integral homology group of~$S_g$.
We prove that the $k$\textsuperscript{th} integral homology group of~$\I_g$ contains a free abelian subgroup of infinite rank, provided that $g\ge 3$ and $2g-3\le k\le 3g-6$. Earlier the same property was known only for $k=3g-5$ (Bestvina, Bux, Margalit, 2007) and in the special case $g=k=3$ (Johnson, Millson, 1992). We also show that  the hyperelliptic involution acts on the constructed infinite system of linearly independent homology classes in $\HH_k(\I_g;\Z)$ as multiplication by~$-1$, provided that $k+g$ is even, thus solving negatively a problem by Hain. For $k=2g-3$, we show that the group~$\HH_{2g-3}(\I_g;\Z)$ contains a free abelian subgroup of infinite rank generated by abelian cycles and we construct explicitly an infinite system of abelian cycles generating such subgroup.  As a consequence of our results, we obtain that an Eilenberg--MacLane CW complex of type~$K(\I_g,1)$ cannot have a finite $(2g-3)$-skeleton.
The proofs are based on the study of the spectral sequence for the action of~$\I_g$ on the complex of cycles constructed by Bestvina, Bux, and Margalit.
\end{abstract}

\maketitle

\section{Introduction}

Let $S_g$ be an oriented closed surface of genus~$g$. Recall that the \textit{mapping class group} of~$S_g$ is the group
$$
\Mod(S_g)=\pi_0\Homeo^+(S_g),
$$
where $\Homeo^+(S_g)$ is the group of orientation preserving homeomorphisms of~$S_g$ on itself. The action of the mapping class group on the first integral homology group of~$S_g$ preserves the intersection form and yields a surjective homomorphism
$$
\Mod(S_g)\to\Sp(2g,\Z).
$$
The kernel of this homomorphism is called the \textit{Torelli group} of~$S_g$ and is denoted by~$\I_g$. 

It is a classical result that $\Mod(S_1)=\mathrm{SL}(2,\Z)$ and so~$\I_1$ is trivial, cf.~\cite[Theorem~2.5]{FaMa12}. Mess~\cite{Mes92} proved that~$\I_2$ is an infinitely generated free group. (The fact that $\I_2$ is not finitely generated was earlier proved by McCullough and Miller~\cite{MCM86}.) Johnson~\cite{Joh83} showed that, for $g\ge 3$, the group~$\I_g$ is finitely generated, and described explicitly a finite set of generators. Nevertheless, the structure of the Torelli groups $\I_g$, where $g\ge 3$, is still rather poorly understood.

An interesting and important problem is to study the homology of the Torelli groups~$\I_g$ for $g\ge 3$. For the sake of simplicity, we denote the homology groups~$\HH_k(G;\Z)$ with integral coefficients simply by~$\HH_k(G)$.
The group $\HH_1(\I_g)$, i.\,e., the abelianization of~$\I_g$, was computed  by Johnson~\cite{Joh85b}. He showed that
$$
\HH_1(\I_g)\cong \Z^{\binom{2g}{3}-2g}\oplus(\Z/2\Z)^{\binom{2g}{2}+2g}
$$
and described explicitly the structure of an $\Sp(2g,\Z)$-module on this group. The problem of explicit computation of~$\HH_k(\I_g)$ for $k\ge 2$ seems to be extremely hard. So the reasonable questions are:
\begin{itemize}
\item Which groups~$\HH_k(\I_g)$ are trivial and which are not?
\item Which groups~$\HH_k(\I_g)$ are finitely generated and which are not?
\end{itemize}
One of the first results towards the latter question was obtained by Akita~\cite{Aki01} who showed that for $g\ge 7$, the total rational homology group
$$
\HH_*(\I_g;\Q)=\bigoplus_{k\ge 0}\HH_k(\I_g;\Q)
$$
is an infinite-dimensional vector space.  For $g=3$, it is known that both groups~$\HH_3(\I_3)$ and~$\HH_4(\I_3)$ contain a free abelian subgroup of infinite rank, see~\cite{Mes92} and~\cite{Hai02}, respectively.

In 2007 Bestvina, Bux, and Margalit~\cite{BBM07} developed a new method for studying the homology of the Torelli groups based on the spectral sequence for the action of~$\I_g$ on a special contractible cell complex~$\B_g$ called  the \textit{complex of cycles}. They proved that the cohomological dimension of~$\I_g$ is equal to~$3g-5$ and the top homology group~$\HH_{3g-5}(\I_g)$ is not finitely generated. (Recall that the \textit{cohomological dimension}~$\cd(G)$ of a group~$G$ is the largest integer~$n$ for which there exists a $G$-module~$M$ such that $\HH^n(G;M)\ne 0$.)  Moreover, the proof in~\cite{BBM07} actually implies that $\HH_{3g-5}(\I_g)$ contains a free abelian subgroup of infinite rank, see Section~\ref{subsection_BBM} for details. The question of whether the groups~$\HH_k(\I_g)$ are finitely generated or not for $1<k<3g-5$ has remained completely open, except for the above mentioned case $k=g=3$ considered in~\cite{Mes92}. The main result of the present paper is as follows.

\begin{theorem}\label{theorem_main}
Suppose that  $g\ge 3$ and $2g-3\le k\le 3g-6$. Then the group $\HH_k(\I_g)$ is not finitely generated. Moreover, $\HH_k(\I_g)$ contains a free abelian subgroup of infinite rank.
\end{theorem}

Recall the definition of an \textit{abelian cycle}. Let $h_1,\ldots,h_k$ be pairwise commuting elements of a group~$G$. Consider the homomorphism
$\chi\colon\Z^k\to G$ that sends the generator of the $i$\textsuperscript{th} factor~$\Z$ to~$h_i$ for every~$i$. We denote by~$\CA(h_1,\ldots,h_k)$ the image of the standard generator~$\mu_k$ of the group~$\HH_k(\Z^k)\cong\Z$ under the homomorphism
$
\chi_*:\HH_k(\Z^k)\to \HH_k(G)
$.
Homology classes $\CA(h_1,\ldots,h_k)$  are called \textit{abelian cycles}. 

Vautaw~\cite{Vau02} proved that the Torelli group~$\I_g$ does not contain a free abelian subgroup of rank greater than~$2g-3$. Since $\I_g$ is torsion-free, this result implies that the group~$\HH_k(\I_g)$ contains no non-trivial abelian cycles, provided that $k>2g-3$. For $k=2g-3$, we can refine Theorem~\ref{theorem_main} in the following way.

\begin{theorem}\label{theorem_Abelian}
Suppose that $g\ge 3$. Then the group $\HH_{2g-3}(\I_g)$ contains a free abelian subgroup of infinite rank  whose generators are abelian cycles.
\end{theorem}

\begin{remark}
For $g=3$, Theorem~\ref{theorem_main} exactly recovers the already mentioned result by Johnson and Millson (unpublished, see~\cite{Mes92}) that $\HH_3(\I_3)$ contains a free abelian subgroup of infinite rank. Nevertheless, Theorem~\ref{theorem_Abelian} seems to be a new result even for genus~$3$. For each pair $(g,k)$ such that  $g\ge 4$ and $2g-3\le k\le 3g-6$, the assertion of Theorem~\ref{theorem_main} is a new result.
\end{remark}

\begin{remark}
Our proof of Theorem~\ref{theorem_Abelian} is constructive, which means that we will construct explicitly an infinite set of linearly independent abelian cycles in $\HH_{2g-3}(\I_g)$. (Recall that elements of an abelian group are \textit{linearly independent} if they form a basis for a free abelian subgroup of this group.) Note also that this construction does not use the result of Bestvina, Bux, and Margalit~\cite{BBM07} claiming that the group $\HH_{3g-5}(\I_g)$ is infinitely generated. On the contrary, our construction of an infinite set of linearly independent homology classes in $\HH_{k}(\I_g)$, where $2g-3<k<3g-5$, is in a sense a mixture of our construction of an infinite set of linearly independent abelian cycles in~$\HH_{2g-3}(\I_g)$ and the Bestvina--Bux--Margalit description of an infinite set of linearly independent homology classes in~$\HH_{3g-5}(\I_g)$. It is a debatable question whether the latter description can be regarded as an explicit construction. Our proof of Theorem~\ref{theorem_main} is constructive modulo this description.
\end{remark}

Bestvina, Bux, and Margalit~\cite{BBM07} obtained their results on the cohomological dimension of~$\I_g$ and on the top homology subgroup~$\HH_{3g-5}(\I_g)$ using a spectral sequence~$E^*_{*,*}$ for the action of~$\I_g$ on a new contractible CW complex~$\B_g$, which they call a \textit{complex of cycles}. Recall that this spectral sequence converges to the homology of~$\I_g$, see Section~\ref{section_CL} for details.
Our proofs of Theorems~\ref{theorem_main} and~\ref{theorem_Abelian} are based on a deeper  study of the same spectral sequence.

\begin{remark}
Bestvina, Bux, and Margalit asked explicitly whether the groups $\HH_k(\I_g)$ are infinitely generated whenever $2g-3\le k\le 3g-6$ (cf. Question~8.1 in~\cite{BBM07}), and suggested a possible approach towards obtaining the affirmative answer to this question. Though our main result (Theorem~\ref{theorem_main}) is exactly the answer to this question, our proof does not follow the approach suggested in~\cite{BBM07}. Namely, Bestvina, Bux, and Margalit showed that the term~$E^1_{0,k}$ of the spectral sequence for the action of~$\I_g$ on~$\B_g$ is infinitely generated, provided that $2g-3\le k\le 3g-6$, and concluded that if one could prove that this group remains infinitely generated after taking consecutive quotients by the images of the differentials $d^1,\ldots,d^{k+1}$, then he would obtain that  $\HH_k(\I_g)$ is infinitely generated, since $E^{\infty}_{0,k}=E^{k+2}_{0,k}$ injects into~$\HH_k(\I_g)$. This plan seems to be very hard to realize because it requires computation of higher differentials of the spectral sequence. Instead, we show that every group~$E^{\infty}_{n,\,3g-5-2n}$, where $1\le n\le g-2$, contains a free abelian subgroup of infinite rank. This also implies that  $\HH_k(\I_g)$ contains a free abelian subgroup of infinite rank for $2g-3\le k\le 3g-6$, since the groups~$E^{\infty}_{p,\,k-p}$, where $p=0,\ldots,k$, are consecutive quotients for certain filtration in~$\HH_k(\I_g)$. The crucial point in our approach is that we avoid explicit computation of higher differentials by mapping the  spectral sequence~$E^*_{*,*}$ to certain auxiliary spectral sequences~$\hE^*_{*,*}(\fN)$ in which the corresponding differentials are trivial by dimension reasons.
\end{remark}

\begin{remark}
 Since the appearance of the first preprint version of this paper, a number of other results have been obtained by means of a detailed study of the spectral sequences for the actions of the Torelli group~$\I_g$ and its subgroups on the complex of cycles~$\B_g$, namely, results on the top homology group of the Johnson kernel~$\mathcal{K}_g$ (see~\cite{Gai22}, \cite{Spi21}) and results on the homology of the genus~$3$ Torelli group~$\I_3$ (see~\cite{Gai21}, \cite{Spi22}).
\end{remark}

Kirby's list of problems in low-dimensional topology~\cite{Kir97} contains the following Problem~2.9(B) attributed to Mess: \textit{Given~$g$, what is the largest~$k$ for which $\I_g$ admits a classifying space with finite $k$-skeleton}. (This is also Problem~5.11 in~\cite{Far06}.) The best previously known estimate $k\le 3g-6$ was obtained by Bestvina, Bux, and Margalit~\cite{BBM07}. The following is a direct consequence of Theorem~\ref{theorem_main}.

\begin{cor}
If $\I_g$ admits a classifying space with finite $k$-skeleton, then $k\le 2g-4$.
\end{cor}

Recall that the \textit{extended Torelli group} of genus~$g$ is the subgroup $\widehat{\I}_g\subset\Mod(S_g)$ consisting of all mapping classes that act on~$\HH_1(S_g)$ either trivially or by multiplication by~$-1$. Obviously, $\widehat{\I}_g/\I_g\cong\Z/2\Z$, which yields the action of the group~$\Z/2\Z$ on the homology of~$\I_g$. If we endow~$S_g$ with the structure of a hyperelliptic complex curve, then the hyperelliptic involution~$s$ will be an element of~$\widehat{\I}_g$ not belonging to~$\I_g$. Hence the action of~$\Z/2\Z$ on~$\HH_*(\I_g)$ is exactly the action of the hyperelliptic involution. If $2$ is invertible in the coefficient ring~$R$, then we obtain the splitting
$$
\HH_*(\I_g;R)=\HH_*(\I_g;R)^+\oplus \HH_*(\I_g;R)^-,
$$
where the generator of~$\Z/2\Z$ acts trivially on~$\HH_*(\I_g;R)^+$ and acts as~$-1$ on~$\HH_*(\I_g;R)^-$.

Hain asked in~\cite[Problem~4.2]{Hai06} whether or not $\HH_*(\I_g;\Z[1/2])^-$ is always a finitely generated $\Z[1/2]$-module. The interest to this question is evoked by the following geometric interpretation of it. Consider the \textit{Torelli space}~$\T_g$, i.\,e., the quotient of the Teichm\"uller space by the (free) action of~$\I_g$. Then $\T_g$ is a~$K(\I_g,1)$, hence, $\HH_*(\T_g;R)=\HH_*(\I_g;R)$. The space~$\T_g$ is the moduli space of compact smooth genus~$g$ complex curves~$C$ together with a symplectic basis in $\HH_1(C)$. Therefore we have the period mapping $\mathcal{P}\colon \T_g\to \h_g$, where $\h_g$ is the upper Siegel half-space. Let $\J_g$ be the image of~$\mathcal{P}$, i.\,e., the set of framed jacobians of smooth genus~$g$ curves. It is well known that for $g\ge 3$, the mapping $\mathcal{P}\colon \T_g\to \J_g$ is two-to-one, branched along the locus of hyperelliptic curves. Let $\sigma$ be the involution on~$\T_g$ interchanging the pre-images of every point in~$\J_g$. It is easy to see that the action of~$\sigma$ on $\HH_*(\T_g;R)$ coincides with the above action of the generator of~$\Z/2\Z=\widehat{\I}_g/\I_g$ on $\HH_*(\I_g;R)$. The mapping $\mathcal{P}$ induces an isomorphism
$$
\HH_*(\I_g;\Z[1/2])^+=\HH_*(\T_g;\Z[1/2])^+\cong  \HH_*(\J_g;\Z[1/2]),
$$
see~\cite[Proposition~5]{Hai02}.
So an equivalent formulation of the above question by Hain is whether the infinite topology of~$\T_g$ localized away from~$2$ entirely comes from~$\J_g$.

It is not hard to show (see Corollary~\ref{cor_invariant}) that the images in $\HH_{3g-5}(\I_g;\Z[1/2])$ of the homology classes constructed by Bestvina, Bux, and Margalit~\cite{BBM07} lie in $\HH_{3g-5}(\I_g;\Z[1/2])^+$. So until now the question by Hain has remained completely open. We give the following negative answer to it.

\begin{theorem}\label{theorem_pm}
 Suppose that $g\ge 3$ and $2g-3\le k\le 3g-6$. 
\begin{enumerate}
\item If $g+k$ is even, then the group~$\HH_k(\I_g;\Z[1/2])^-$ contains an infinitely generated free $\Z[1/2]$-module.
\item If $g+k$ is odd, then the group $\HH_k(\J_g;\Z[1/2])\cong \HH_k(\I_g;\Z[1/2])^+$ contains an infinitely generated free $\Z[1/2]$-module.
\end{enumerate}
\end{theorem}

This paper is organized as follows. In Section~\ref{section_explicit} we describe explicitly an infinite linearly independent system of abelian cycles in~$\HH_{2g-3}(\I_g)$ (Theorem~\ref{theorem_Abelian_explicit}) and an infinite linearly independent system of homology classes in~$\HH_{3g-5-n}(\I_g)$ for every $n=1,\ldots,g-3$ (Theorem~\ref{theorem_explicit}). Also in  Section~\ref{section_explicit}, we show that  these homology classes localized away from~$2$ lie in~$\HH_k(\I_g;\Z[1/2])^-$ whenever $g+k$ is even and in~$\HH_k(\I_g;\Z[1/2])^+$ whenever $g+k$ is odd.
The rest part of the paper contains the proof that the constructed systems of homology classes are indeed linearly independent. 
The main tool is the spectral sequence for the action of~$\I_g$ on the complex of cycles~$\B_g$.
Sections~\ref{section_cycles} and~\ref{section_CL} contain necessary background information on the complex of cycles~$\B_g$ and on the spectral sequence for the action of a group on a CW complex, respectively. In Section~\ref{section_lemma} we prove a useful lemma explaining how generalized abelian cycles look like from the viewpoint of this spectral sequence.
In the last five sections we study the spectral sequence~$E^*_{*,*}$ for the action of~$\I_g$ on~$\B_g$.
In Section~\ref{section_aux_spectral} we construct auxiliary spectral sequences~$E^*_{*,*}(\fN)$ and~$\hE^*_{*,*}(\fN)$ and morphisms of spectral sequences 
$$
E^*_{*,*}\to E^*_{*,*}(\fN)\to \hE^*_{*,*}(\fN).
$$
This construction is needed to avoid explicit computation of higher differentials of the spectral sequence~$E^*_{*,*}$. The auxiliary spectral sequences~$E^*_{*,*}(\fN)$ and~$\hE^*_{*,*}(\fN)$ depend on the choice of an $\I_g$-orbit~$\fN$ of oriented multicurves on~$S_g$. We will need to study these spectral sequences for certain special oriented multicurves, which we call $\alpha$-\textit{multicurves}. In Section~\ref{section_alpha} we define  $\alpha$-multicurves and prove some their properties. In Section~\ref{section_alpha_E} we study the properties of the spectral sequence~$\hE^*_{*,*}(\fN)$ in the case when $\fN$ is an $\I_g$-orbit consisting of $\alpha$-multicurves. Finally, in Sections~\ref{section_Abelian_proof} and~\ref{section_main_proof} we prove Theorems~\ref{theorem_Abelian_explicit} and~\ref{theorem_explicit}, respectively.

\smallskip

The author is grateful to S.~I.~Adian and A.~L.~Talambutsa for fruitful discussions.

\section{Systems of linearly independent homology classes}\label{section_explicit}

A \textit{multicurve} on the surface~$S_g$ is a union of a finite number of pairwise disjoint simple closed curves $M=\gamma_1\cup\cdots\cup\gamma_k$ such that no curve~$\gamma_i$ is homotopic to a point and no two curves~$\gamma_i$ and~$\gamma_j$ are homotopic to each other. A multicurve~$M$ is said to be \textit{oriented} if all components of~$M$ are endowed with orientations. (Note that, by definition, an oriented multicurve is not allowed to contain a pair of components that become homotopic to each other after reversing the orientation of one of them.)

Most constructions in the present paper concerning curves and multicurves are determined up to isotopy. For the sake of simplicity, we will later not distinguish between notation for a curve (or a multicurve) and its isotopy class. For instance, the subgroup of~$\I_g$ consisting of all mapping classes~$h$ that stabilize the isotopy class of a multicurve~$M$  will be denoted by $\Stab_{\I_g}(M)$ and will be called the stabilizer of the multicurve~$M$.

We denote by $[\alpha]$ the integral homology class of an oriented simple closed curve~$\alpha$. We denote by $a\cdot b$ the algebraic intersection index of homology classes $a,b\in \HH_1(S_g)$. We denote by $\langle c_1,\ldots,c_k\rangle$ the subgroup of~$\HH_1(S_g)$ generated by homology classes $c_1,\ldots,c_k$.  We denote by~$T_{\alpha}$ the left Dehn twist about a simple closed curve~$\alpha$. A \textit{bounding pair} is a pair~$\{\alpha,\alpha'\}$ of disjoint non-separating not homotopic to each other simple closed curves in the same homology class. The mapping class $T_{\alpha}T_{\alpha'}^{-1}$ will be called a \textit{bounding pair map} or simply a \textit{BP map}. Recall that Dehn twists about separating simple closed curves and BP maps belong to the Torelli group.

We choose a primitive element $x\in \HH_1(S_g)$ and fix it throughout the whole paper.

Further we always suppose that $g\ge 2$.

\subsection{System of linearly independent abelian cycles}\label{subsection_A}
Consider a  splitting
\begin{equation}\label{eq_splitting}
\HH_1(S_g)=U_0\oplus\cdots\oplus U_{g-1},
\end{equation}
where every~$U_j$ has rank~$2$ and $U_i$ is orthogonal to~$U_j$ with respect to the intersection form unless~$i=j$. Note that these conditions immediately imply that the restriction of the intersection form to each~$U_i$ has the matrix 
$$
\left(
\begin{aligned}
0&&1&\\-1&&0&
\end{aligned}
\right)
$$
in a suitable basis.

For each splitting of the form~\eqref{eq_splitting}, there is a unique decomposition
\begin{equation}\label{eq_x_decompose}
x=x_0+\cdots+x_{g-1},\qquad x_i\in U_i.
\end{equation}
Let $\CS$ be the set of all  splittings of the form~\eqref{eq_splitting} such that all summands~$x_i$ in the corresponding decomposition~\eqref{eq_x_decompose} are nonzero.  Splittings obtained from each other by permutations of summands are regarded to be different, i.\,e., a splitting is an ordered $g$-tuple $\CU=(U_0,\ldots,U_{g-1})$. 

The following proposition is straightforward.

\begin{propos}\label{propos_infinite1}
The set~$\CS$ is infinite.
\end{propos}

Consider a splitting $\CU=(U_0,\ldots,U_{g-1})$ belonging to~$\CS$. Every summand in the decomposition~\eqref{eq_x_decompose} can be uniquely written as $x_i=l_ia_i$, where $a_i$ is primitive and~$l_i>0$. Then $a_0,\ldots,a_{g-1}$ is a basis of a \textit{Lagrangian subspace} of~$\HH_1(S_g)$, i.\,e., a rank $g$ direct summand such that the intersection form restricts trivially to it. We denote by $\fB(\CU)$ the set of all $(g-2)$-tuples $\bb=(b_1,\ldots,b_{g-2})$ such that $b_i\in U_i$ and $a_i\cdot b_i=1$ for every $i=1,\ldots,g-2$. (Note that we do not  choose elements~$b_0$ and~$b_{g-1}$  in $U_0$ and~$U_{g-1}$.)

\begin{definition}\label{defin_delta}
A \textit{$\delta$-multicurve} for~$\CU$ is a multicurve 
$$
\Delta=\delta_1\cup\cdots\cup\delta_{g-1}
$$
that satisfies the following conditions:
 \begin{enumerate}
 \item $\delta_1,\ldots,\delta_{g-1}$ are separating simple closed curves,
 \item $S_g\setminus\Delta$ consists of $g$ connected components $X_0\ldots,X_{g-1}$ such that
\begin{itemize} 
\item $X_0$ and~$X_{g-1}$ are once-punctured tori adjacent to~$\delta_1$ and to~$\delta_{g-1}$, respectively,
 \item for every $i=1,\ldots,g-2$, the component~$X_i$ is a twice-punctured torus adjacent to~$\delta_{i}$ and~$\delta_{i+1}$,
 \end{itemize}
 \item the image of the homomorphism $\HH_1(X_i)\to \HH_1(S_g)$ induced by the inclusion is~$U_i$.
 \end{enumerate}
\end{definition}

\begin{definition}\label{defin_beta}
A \textit{$\beta$-multicurve} for $\CU$ compatible with a $\delta$-multicurve~$\Delta$ is an oriented multicurve 
$$
B=\beta_1\cup\beta_1'\cup\cdots\cup\beta_{g-2}\cup\beta_{g-2}'
$$
that satisfies the following conditions:
\begin{enumerate}
\item $\beta_i$ and $\beta_i'$ are homologous  to each other oriented simple closed curves contained in~$X_i$,
\item the $(g-2)$-tuple $\bb=\bigl([\beta_1],\ldots,[\beta_{g-2}]\bigr)$ belongs to~$\fB(\CU)$,
\end{enumerate}
see Fig.~\ref{fig_beta_delta}. The union $\Gamma=B\cup\Delta$ will be called a \textit{$\beta\delta$-multicurve} for~$\CU$.
\end{definition}

\begin{figure}
\begin{tikzpicture}[scale=.8]

\draw[mygreen, thick] (1.5,-1.5) arc (270:90:0.3 and 1.5)
node[pos=0,below]{$\delta_1$};
\draw[mygreen, thick, dashed] (1.5,-1.5) arc (-90:90:0.3 and 1.5);

\draw[mygreen, thick] (4.5,-1.5) arc (270:90:0.3 and 1.5)
node[pos=0,below]{$\delta_2$};
\draw[mygreen, thick, dashed] (4.5,-1.5) arc (-90:90:0.3 and 1.5);

\draw[mygreen, thick] (7.5,-1.5) arc (270:90:0.3 and 1.5)
node[pos=0,below]{$\delta_3$};
\draw[mygreen, thick, dashed] (7.5,-1.5) arc (-90:90:0.3 and 1.5);

\draw[mygreen, thick] (10.5,-1.5) arc (270:90:0.3 and 1.5)
node[pos=0,below]{\,\,\,\,\,\,$\delta_{g-2}$};
\draw[mygreen, thick, dashed] (10.5,-1.5) arc (-90:90:0.3 and 1.5);

\draw[mygreen, thick] (13.5,-1.5) arc (270:90:0.3 and 1.5)
node[pos=0,below]{\,\,\,\,\,\,$\delta_{g-1}$};
\draw[mygreen, thick, dashed] (13.5,-1.5) arc (-90:90:0.3 and 1.5);

\draw[mygreen, thick, snake=brace] (-1.4,1.75) -- (1.4,1.75) 
node[pos=.5,above]{$X_0$};
\draw[mygreen, thick, snake=brace] (1.6,1.75) -- (4.4,1.75)
node[pos=.5,above]{$X_1$};
\draw[mygreen, thick, snake=brace] (4.6,1.75) -- (7.4,1.75)
node[pos=.5,above]{$X_2$};
\draw[mygreen, thick, snake=brace] (10.6,1.75) -- (13.4,1.75)
node[pos=.5,above]{$X_{g-2}$};
\draw[mygreen, thick, snake=brace] (13.6,1.75) -- (16.4,1.75)
node[pos=.5,above]{$X_{g-1}$};

\draw[myblue, thick] (3,1.5) arc (90:270:.15 and .55); 
\draw[myblue, thick,->] (3,1.5) arc (90:185:.15 and .55); 
\draw[myblue, thick, dashed] (3,1.5) arc (90:-90:.15 and .55)
node[pos=0.5,right]{\small{${}\!\beta_{1}$}}; 

\draw[myblue, thick] (3,-.4) arc (90:270:.15 and .55); 
\draw[myblue, thick,->] (3,-1.5) arc (270:175:.15 and .55); 
\draw[myblue, thick, dashed] (3,-.4) arc (90:-90:.15 and .55)
node[pos=0.5,right]{\small{${}\!\beta_{1}'$}}; 

\draw[myblue, thick] (6,1.5) arc (90:270:.15 and .55); 
\draw[myblue, thick,->] (6,1.5) arc (90:185:.15 and .55); 
\draw[myblue, thick, dashed] (6,1.5) arc (90:-90:.15 and .55)
node[pos=0.5,right]{\small{${}\!\beta_{2}$}}; 

\draw[myblue, thick] (6,-.4) arc (90:270:.15 and .55); 
\draw[myblue, thick,->] (6,-1.5) arc (270:175:.15 and .55); 
\draw[myblue, thick, dashed] (6,-.4) arc (90:-90:.15 and .55)
node[pos=0.5,right]{\small{${}\!\beta_{2}'$}}; 

\draw[myblue, thick] (12,1.5) arc (90:270:.15 and .55); 
\draw[myblue, thick,->] (12,1.5) arc (90:185:.15 and .55); 
\draw[myblue, thick, dashed] (12,1.5) arc (90:-90:.15 and .55)
node[pos=0.5,right]{\small{${}\!\beta_{g-2}$}}; 

\draw[myblue, thick] (12,-.4) arc (90:270:.15 and .55); 
\draw[myblue, thick,->] (12,-1.5) arc (270:175:.15 and .55); 
\draw[myblue, thick, dashed] (12,-.4) arc (90:-90:.15 and .55)
node[pos=0.5,right]{\small{${}\!\beta_{g-2}'$}};

\draw[black, ultra thick] (0,1.5) -- (15,1.5) arc (90:-90:1.5) -- (0,-1.5) arc (270:90:1.5);
\draw[black, ultra thick] (0,0) circle (.4);
\draw[black, ultra thick] (3,0) circle (.4);
\draw[black, ultra thick] (6,0) circle (.4);
\draw[black, ultra thick] (12,0) circle (.4);
\draw[black, ultra thick] (15,0) circle (.4);

\fill[black] (9,0) circle (1.5pt);
\fill[black] (9.3,0) circle (1.5pt);
\fill[black] (8.7,0) circle (1.5pt);

\end{tikzpicture}
\caption{A $\beta\delta$-multicurve for $\CU\in\CS$}\label{fig_beta_delta}
\end{figure}

From the fact that the group $\Sp(2g,\Z)$ acts transitively on the set of all orthogonal splittings $\HH_1(S_g)=U_0\oplus\cdots\oplus U_{g-1}$ with $\rank U_i=2$ it follows that a $\delta$-multicurve exists  for any splitting~$\CU\in\CS$. For each primitive homology class in~$U_i$, there exists a pair of non-isotopic simple closed curves in~$X_i$ in this homology class. Hence, for any pair~$(\CU,\Delta)$, there exists a $\beta$-multicurve for~$\CU$ compatible with~$\Delta$ with any prescribed set of homology classes of components $\bb\in\fB(\CU)$.

We conveniently agree that we always choose which of the two components of~$B$ contained in~$X_i$ is~$\beta_i$ and which is~$\beta'_i$ so that the following condition is satisfied:
 the connected component of~$X_i\setminus(\beta_i\cup\beta_i')$ that is adjacent to~$\delta_i$ lies on the right-hand side from~$\beta_i$ and on the left-hand side from~$\beta'_i$.

\begin{propos}\label{propose_one_orbit}
Two $\beta\delta$-multicurves for~$\CU$ lie in the same $\I_g$-orbit if and only if the corresponding $(g-2)$-tuples $\bb\in\fB(\CU)$ coincide.
\end{propos}

\begin{proof}
The `only if' part of the proposition is obvious; let us prove the `if' part. Suppose that $\Gamma$ and~$\widetilde{\Gamma}$ are two $\beta\delta$-multicurves for~$\CU$ with the same~$\bb=(b_1,\ldots,b_{g-2})$. Using the Alexander method (cf.~\cite[Section~2.3]{FaMa12}), one easily obtains that there exists a mapping class~$f\in\Mod(S_g)$ that takes $\Gamma$ to~$\widetilde{\Gamma}$ preserving the orientations of the components~$\beta_i$ and~$\beta_i'$. Consider the action of~$f$ on~$\HH_1(S_g)$. We have $f(U_i)=U_i$ for $i=0,\ldots,g-1$ and $f(b_i)=b_i$ for $i=1,\ldots,g-2$.  If $1\le i\le g-2$, then $a_i,b_i$ is a symplectic basis of~$U_i$. So  $f(a_i)=a_i+m_ib_i$ for some $m_i\in\Z$. Therefore, the mapping class $$f_1=fT_{\beta_1}^{-m_1}\cdots T_{\beta_{g-2}}^{-m_{g-2}}$$ still takes~$\Gamma$ to~$\widetilde{\Gamma}$ and acts trivially on each of the summands $U_1,\ldots,U_{g-2}$.  On each of the rest two summands~$U_0$ and~$U_{g-1}$ the mapping class~$f_1$ acts by a matrix in~$\mathrm{SL}(2,\Z)$. The corresponding components~$X_0$ and~$X_{g-1}$ of $S_g\setminus\Gamma$ are once-punctured tori. Hence there exists an orientation preserving homeomorphism $\varphi\colon S_g\to S_g$ that is the identity on $S_g\setminus(X_0\cup X_{g-1})$ and acts on $U_0$ and~$U_{g-1}$ in the same way as~$f_1$ does. Then $f_2=f_1\varphi^{-1}$ is a required element of the Torelli group~$\I_g$ taking~$\Gamma$ to~$\widetilde{\Gamma}$.
\end{proof}

For each $\CU\in\CS$ and each $\bb\in\fB(\CU)$, we consider the abelian cycle
\begin{equation*}
A_{\CU,\bb}=\CA\bigl(
T_{\beta_1}T_{\beta_1'}^{-1},\ldots,T_{\beta_{g-2}}T_{\beta_{g-2}'}^{-1}, T_{\delta_1},\ldots, T_{\delta_{g-1}}
\bigr)\in \HH_{2g-3}(\I_g),
\end{equation*}
where 
$$
\Gamma =\beta_1\cup\beta_1'\cup\cdots\cup\beta_{g-2}\cup\beta_{g-2}'\cup \delta_1\cup\cdots\cup\delta_{g-1}
$$
is a $\beta\delta$-multicurve for~$\CU$ with $\bigl([\beta_1],\ldots,[\beta_{g-2}]\bigr)=\bb$. It follows from Proposition~\ref{propose_one_orbit} that the homology class~$A_{\CU,\bb}$ is independent of the choice of~$\Gamma$.

For $\CU=(U_0,\ldots,U_{g-1})$ and $\bb=(b_1,\ldots,b_{g-2})$, we  denote by 
$\overline{\CU}$ and $\overline{\bb}$ the tuples obtained from~$\CU$ and~$\bb$ by reversing the orders of~$U_j$'s and $b_j$'s, respectively, i.\,e., $\overline{\CU}=(U_{g-1},\ldots,U_0)$, $\overline{\bb}=(b_{g-2},\ldots,b_{1})$. We have,
$$
A_{\overline{\CU},\overline{\bb}}=A_{\CU,\bb}.
$$
Indeed, both of these abelian cycles come from the same multicurve~$\Gamma$, so we only need to take care of the sign. Reversing the order of the Dehn twists~$T_{\delta_i}$ gives factor~$(-1)^{\binom{g-1}{2}}$, and reversing the order of the BP maps~$T_{\beta_i}T_{\beta_i'}^{-1}$ gives factor~$(-1)^{\binom{g-2}{2}}$. Besides, our convention on which of the two components in~$X_i$ is $\beta_i$ and which is~$\beta_i'$ leads to the fact that each pair of such components must be swapped when we pass from~$(\CU,\bb)$ to~$(\overline{\CU},\overline{\bb})$, which leads to additional factor~$(-1)^{g-2}$. The product of these three factors is equal to~$1$.

\begin{theorem}\label{theorem_Abelian_explicit}
Let $\CS'\subset\CS$ be a subset containing exactly one splitting in every pair~$\{\CU,\overline{\CU}\}$. Choose any representatives $\bb_{\CU}\in\fB(\CU)$, where $\CU\in \CS'$. Then the abelian cycles~$A_{\CU,\bb_{\CU}}$, where $\CU$ runs over~$\CS'$, are linearly independent in $\HH_{2g-3}(\I_g)$.
\end{theorem}

Since the set $\CS$ is infinite,  Theorem~\ref{theorem_Abelian} follows from Theorem~\ref{theorem_Abelian_explicit}. We will prove Theorem~\ref{theorem_Abelian_explicit} in Section~\ref{section_Abelian_proof}.

We denote the image of~$A_{\CU,\bb}$ under the natural mapping $\HH_{2g-3}(\I_g)\to \HH_{2g-3}(\I_g;\Z[1/2])$ again by~$A_{\CU,\bb}$.

\begin{propos}\label{propos_pm_abelian}
The classes $A_{\CU,\bb}$ lie in $\HH_{2g-3}(\I_g;\Z[1/2])^-$ whenever $g$ is odd and lie in $\HH_{2g-3}(\I_g;\Z[1/2])^+$ whenever $g$ is even.
\end{propos}

\begin{proof}
Let $s$ be the rotation of~$S_g$  around the horizontal axis by~$\pi$ that stabilizes every curve~$\delta_i$ and swaps the curves in every pair~$\{\beta_i,\beta_{i}'\}$ (cf. Fig.~\ref{fig_s} in Subsection~\ref{subsection_BBM}). Then~$s$ acts on~$\HH_1(S_g)$ as~$-1$, hence, $s\in\widehat{\I}_g$ and~$s\notin\I_g$. Therefore $\HH_{2g-3}(\I_g;\Z[1/2])^{\pm}$ are the eigenspaces with eigenvalues $\pm 1$ for the action of~$s$ on~$\HH_{2g-3}(\I_g;\Z[1/2])$. We have
\begin{equation*}
s_*(A_{\CU,\bb})=\CA\bigl(
T_{\beta_1'}T_{\beta_1}^{-1},\ldots,T_{\beta_{g-2}'}T_{\beta_{g-2}}^{-1}, T_{\delta_1},\ldots, T_{\delta_{g-1}}
\bigr)=(-1)^{g-2}A_{\CU,\bb},
\end{equation*}
which implies the required assertion.
\end{proof}

The assertion of Theorem~\ref{theorem_pm} for $k=2g-3$ follows immediately from Theorem~\ref{theorem_Abelian_explicit} and Proposition~\ref{propos_pm_abelian}.

\subsection{Bestvina--Bux--Margalit classes}\label{subsection_BBM}
Let $S_{g,1}$ be an oriented genus~$g$ surface with one puncture and let $S_g^1$ be an oriented compact genus~$g$ surface with one boundary component. Consider the corresponding mapping class groups
\begin{align*}
	\Mod(S_{g,1})&=\pi_0\Homeo^+(S_{g,1}),\\
	\Mod(S_g^1)&=\pi_0\Homeo(S_g^1,\partial S_g^1),
\end{align*} 
where $\Homeo(S_g^1,\partial S_g^1)$ is the group of all homeomorphisms $S_g^1\to S_g^1$ that are the identity on the boundary.
Let $\I_{g,1}$ and~$\I_g^1$ be the corresponding Torelli groups, that is, the kernels of the natural surjective homomorphisms $\Mod(S_{g,1})\to\Sp(2g,\Z)$ and $\Mod(S_g^1)\to\Sp(2g,\Z)$, respectively. One more group, which will be of our interest, is the stabilizer $\Stab_{\I_{g+1}}(v)$ of a non-separating simple closed curve~$v$ on~$S_{g+1}$.

\begin{theorem}[Bestvina, Bux, Margalit~\cite{BBM07}]\label{theorem_BBM2}
 Suppose that $g\ge 2$. Then
 \begin{align*}
 	\cd(\I_g)=3g-5,\qquad\qquad
 	\cd(\I_{g,1})&=3g-3,\qquad\qquad
 	\cd(\I_g^1)=3g-2,\\
 	\cd\bigl(\Stab_{\I_{g+1}}(v)\bigr)&=3g-2,
 \end{align*}
 and the top homology groups~$\HH_{3g-5}(\I_g)$, $\HH_{3g-3}(\I_{g,1})$, $\HH_{3g-2}(\I_g^1)$, and~$\HH_{3g-2}\bigl(\Stab_{\I_{g+1}}(v)\bigr)$ contain  free abelian subgroups of infinite rank.
\end{theorem}

\begin{remark}
	In~\cite{BBM07} this theorem is formulated for~$\I_g$, $\I_{g,1}$, and~$\Stab_{\I_{g+1}}(v)$  while the group~$\I_g^1$ is never mentioned. Nevertheless, the result for~$\I_g^1$ follows immediately from the result for~$\I_{g,1}$ and in fact is implicitly contained in~\cite{BBM07}. Besides, Bestvina, Bux, and Margalit asserted only that the top homology groups are not finitely generated. However, their proof in fact yields that these groups contain free abelian subgroups of infinite rank.
\end{remark}

Let us now recall in convenient for us form the construction from~\cite{BBM07} of an infinite number of linearly independent homology classes in the top homology groups of the Torelli groups~$\I_g$, $\I_{g,1}$ and~$\I_g^1$ and the stabilizer~$\Stab_{\I_{g+1}}(v)$. We have to slightly change the description of this construction, since for us the case of~$\I_g^1$, which was omitted in~\cite{BBM07}, is of the most interest. The construction consists of the following sequence of injective homomorphisms (for $g\ge 2$), the first two of which are isomorphisms:
\begin{equation}\label{eq_mono_sequence}
\HH_{3g-5}(\I_g)\xrightarrow{\cong} \HH_{3g-3}(\I_{g,1})\xrightarrow{\cong} \HH_{3g-2}(\I_g^1)\hookrightarrow \HH_{3g-2}\bigl(\Stab_{\I_{g+1}}(v)\bigr)\hookrightarrow \HH_{3g-2}(\I_{g+1}).
\end{equation}
Let us describe these homomorphisms one by one and explain how to extract their injectivity from~\cite{BBM07}.

1. The following short exact sequence  is due to Johnson~\cite{Joh83} (cf.~\cite[Lemma~8.3]{BBM07}):
\begin{equation}\label{eq_ses}
1\to \pi_1(S_g)\to \I_{g,1}\to\I_g\to 1.
\end{equation}
Since $\cd(\I_g)=3g-5$, $\cd\bigl( \pi_1(S_g)\bigr)=2$, and the group~$\I_g$ acts trivially on the homology group $\HH_2(\pi_1(S_g))=\HH_2(S_g)=\Z$, the Lyndon--Hochschild--Serre spectral sequence~$E^*_{*,*}$ for this short exact sequence yields the first isomorphism in the sequence~\eqref{eq_mono_sequence}:
$$
\HH_{3g-5}(\I_g)=\HH_{3g-5}(\I_g,\HH_2(\pi_1(S_g)))=E^2_{3g-5,\,2}=E^{\infty}_{3g-5,\,2}\cong \HH_{3g-3}(\I_{g,1}),
$$
see~\cite[Lemma~8.8]{BBM07} for details.

2. An inclusion $S_g^1\subset S_{g,1}$ induces a short exact sequence
\begin{equation}\label{eq_exact}
	1\to\Z\to\I_g^1\to\I_{g,1}\to 1.
\end{equation}
Here the subgroup~$\Z\subset\I_g^1$ is generated by the Dehn twist about a simple curve isotopic to the boundary of~$S_g^1$, hence, is central in~$\I_g^1$. Since $\cd(\I_{g,1})=3g-3$, the Lyndon--Hochschild--Serre spectral sequence for the central extension~\eqref{eq_exact} yields the second isomorphism in the sequence~\eqref{eq_mono_sequence}, namely,
$\HH_{3g-3}(\I_{g,1})\cong \HH_{3g-2}(\I_g^1)$.

3. Let $\varphi\colon S^1_g\hookrightarrow S_{g+1}\setminus v$ be an arbitrary embedding. Extending homeomorphisms by the identity yields a homomorphism $\I_g^1\to\Stab_{\I_{g+1}}(v)$. For the third homomorphism in the sequence~\eqref{eq_mono_sequence}, we take the corresponding homomorphism in homology~$\varphi_*\colon \HH_{3g-2}(S^1_g)\to \HH_{3g-2}\bigl(\Stab_{\I_{g+1}}(v)\bigr)$. For the injectivity of~$\varphi_*$, see Proposition~\ref{propos_injection} below.

4. The fourth homomorphism in the sequence~\eqref{eq_mono_sequence} is the homomorphism induced by the inclusion $\Stab_{\I_{g+1}}(v)\subset\I_{g+1}$. It is injective by~\cite[Lemma~8.10]{BBM07}.

\begin{propos}\label{propos_injection}
Let $v$ be a non-separating simple closed curve on~$S_{g+1}$, where $g\ge 1$. Then any embedding $\varphi\colon S_g^1\hookrightarrow S_{g+1}\setminus v$ induces an injective homomorphism $\HH_{3g-2}(\I_g^1)\to \HH_{3g-2}\bigl(\Stab_{\I_{g+1}}(v)\bigr)$.
\end{propos}

\begin{proof}
The proof of this proposition is essentially contained in the proofs of Lemma~8.7 and Theorem~C in~\cite{BBM07}. Let us explain how to extract the proof from what is written there.

We have a commutative diagram with exact rows:
\begin{equation}\label{eq_CD}
\begin{tikzcd}
1 \arrow[r] &
\Z  \arrow[r] \arrow[d, "i"] &
\I_g^1 \arrow [r]\arrow [d,"\varphi_*"] &
\I_{g,1} \arrow[r] \arrow[d, equal] &
1\\
1 \arrow[r] &
K  \arrow[r] &
\Stab_{\I_{g+1}}(v) \arrow[r] &
\I_{g,1} \arrow[r] &
1
\end{tikzcd}
\end{equation}
Here the first row is the exact sequence~\eqref{eq_exact}, the second row is the restriction of the Birman exact sequence obtained in~\cite[Lemma~8.3]{BBM07}, and $K$ is the commutator subgroup of~$\pi_1(S_{g,1})$, where $S_{g,1}$ is a once-punctured surface of genus~$g$ obtained from~$S_{g+1}\setminus v$ by forgetting one of the two punctures. The group~$\Z$ in the first row is generated by the Dehn twist~$T_{\varepsilon}$, where $\varepsilon$ is a simple closed curve isotopic to the boundary of~$S_g^1$. By~\cite[Lemma~8.6]{BBM07}, the element~$i(T_{\varepsilon})$ represents a non-trivial homology class in~$\HH_1(K)$. But the group~$K$ is free, since it is a subgroup of the free group~$\pi_1(S_{g,1})$. Hence $\HH_1(K)$ is free abelian. Therefore, the homomorphism
$
i_*\colon\Z=\HH_1(\Z)\to \HH_1(K)
$
is an injection. This injection induces a homomorphism
\begin{equation}\label{eq_injection_i*}
i_*\colon \HH_{3g-3}(\I_{g,1})=\HH_{3g-3}(\I_{g,1},\HH_1(\Z))\to \HH_{3g-3}(\I_{g,1},\HH_1(K)),
\end{equation}
where $\HH_1(K)$ is endowed with the structure of an $\I_{g,1}$-module corresponding to the second row of~\eqref{eq_CD}. (Since $T_{\varepsilon}$ lies in the centre of~$\I_{g,1}$, the action of~$\I_{g,1}$ on the coefficient group $\HH_1(\Z)=\Z$ is trivial.) It is shown in the proof of Theorem~C in~\cite{BBM07} that the homomorphism~\eqref{eq_injection_i*} is an injection. (Note that the submodule~$M\subset \HH_1(K)$ that is constructed in~\cite[Lemma~8.7]{BBM07} and used in the proof of Theorem~C in~\cite{BBM07} is exactly~$i_*(\Z)$.)

Finally, consider the Lyndon--Hochschild--Serre spectral sequences for the rows of~\eqref{eq_CD}. Since  $\cd(\I_{g,1})=3g-3$ and $\cd(\Z)=\cd(K)=1$, these spectral sequences give a commutative diagram
\begin{equation*}
\begin{tikzcd}
\HH_{3g-2}(\I_g^1) \arrow[r, leftrightarrow, "\cong" above]
\arrow[d, "\varphi_*"]
& \HH_{3g-3}(\I_{g,1},\Z) \arrow[d,"i_*" left, hookrightarrow]\\
\HH_{3g-2}\bigl(\Stab_{\I_{g+1}}(v)\bigr) \arrow[r, leftrightarrow, "\cong" above]
& \HH_{3g-3}(\I_{g,1},\HH_1(K))
\end{tikzcd}
\end{equation*}
Thus, $\varphi_*$ is an injection.
\end{proof}

Recall that, by a result of Mess~\cite{Mes92}, the group $\I_2$ is an infinitely generated free group and hence $\HH_1(\I_2)$ is an infinitely generated free abelian group. Iterated application of homomorphisms in sequences~\eqref{eq_mono_sequence} provides injections of the group~$\HH_1(\I_2)$ into the groups~$\HH_{3g-5}(\I_g)$, $\HH_{3g-3}(\I_{g,1})$, $\HH_{3g-2}(\I_g^1)$, and~$\HH_{3g-2}\bigl(\Stab_{\I_{g+1}}(v)\bigr)$ for any~$g$.

\begin{definition}
Homology classes in the groups~$\HH_{3g-5}(\I_g)$, $\HH_{3g-3}(\I_{g,1})$, $\HH_{3g-2}(\I_g^1)$, and~$\HH_{3g-2}\bigl(\Stab_{\I_{g+1}}(v)\bigr)$ that can be obtained from homology classes in $\HH_1(\I_2)$ by an iterated application of homomorphisms in sequences~\eqref{eq_mono_sequence} will be called \textit{Bestvina--Bux--Margalit} classes.
\end{definition}

We see that each of the top homology groups~$\HH_{3g-5}(\I_g)$, $\HH_{3g-3}(\I_{g,1})$, $\HH_{3g-2}(\I_g^1)$, and~$\HH_{3g-2}\bigl(\Stab_{\I_{g+1}}(v)\bigr)$ contains an infinitely generated free abelian subgroup consisting of Bestvina--Bux--Margalit classes.

\begin{remark}
 The homomorphisms in the sequence~\eqref{eq_mono_sequence} are not canonical. Indeed, they depend on the choice of a non-separating simple closed curve $v$ and an embedding $\varphi\colon S^1_g\hookrightarrow S_{g+1}\setminus v$. The relations between the Bestvina--Bux--Margalit classes corresponding to different choices of these data are yet to be studied.
\end{remark}

Let us now show that Bestvina--Bux--Margalit homology classes in~$\HH_{3g-5}(\I_g)$  lie in $\HH_{3g-5}(\I_g;\Z[1/2])^+$ after localization away from~$2$, and similarly for $\HH_{3g-3}(\I_{g,1})$, $\HH_{3g-2}(\I_g^1)$, and~$\HH_{3g-2}\bigl(\Stab_{\I_{g+1}}(v)\bigr)$.

As in the closed case, the \textit{extended Torelli group}~$\widehat{\I}_{g,1}$ consists of all mapping classes in~$\Mod(S_{g,1})$ that act on~$\HH_1(S_{g,1})$   either trivially or as~$-1$. In the case of~$S_{g}^{1}$, we could define an extended Torelli group likewise, that is, as the subgroup of~$\Mod(S^1_g)$ consists of all mapping classes that act on~$\HH_1(S_g^1)$ either trivially or as~$-1$. Nevertheless, we prefer to refer to another group as the \textit{extended Torelli group} in this case. Namely, let~$\widehat{\I}_g^1$ be the group consisting of all isotopy classes of orientation preserving homeomorphisms $f\colon S_g^1\to S_g^1$ such that either $f$ acts trivially on~$\HH_1(S_g^1)$ and fixes $\partial S_g^1$ pointwise or  $f$ acts as~$-1$ on~$\HH_1(S_g^1)$ and acts as rotation by~$\pi$ on~$\partial S_g^1$. An advantage of such extension of the Torelli group~$\I_g^1$ is that $\widehat{\I}^1_g$ contains a \textit{hyperelliptic involution}, which is an element of order~$2$ in $\widehat{\I}_g^1$ not belonging to~$\I_g^1$. As in the closed case, the groups~$\I_{g,1}$ and $\I_{g}^1$ are subgroups of index~$2$ of the groups~$\widehat{\I}_{g,1}$ and~$\widehat{\I}_{g}^1$, respectively, which yields the action of~$\Z/2\Z$ on the homology of~$\I_{g,1}$ and $\I_{g}^1$.

Note also that $\Stab_{\I_{g+1}}(v)$  is a subgroup of index~$2$ of the group~$\Stab_{\widehat{\I}_{g+1}}(v)$, since there exists a mapping class in~$\widehat{\I}_{g+1}$ that takes $v$ to itself reversing the orientation of it. This yields the action of $\Z/2\Z$ on the homology of~$\Stab_{\I_{g+1}}(v)$.

\begin{propos}\label{propos_equivariant}
All homomorphisms in the sequence~\eqref{eq_mono_sequence}
are $\Z/2\Z$-equivariant.
\end{propos}

\begin{proof}
Recall that the sequence of homomorphisms~\eqref{eq_mono_sequence} depends on the choice of a non-separating simple closed curve~$v$ on~$S_{g+1}$ and an embedding $\varphi\colon S_g^1\hookrightarrow S_{g+1}\setminus v$. We identify~$S_g^1$ with its image under~$\varphi$ and arrange the surface~$S_{g+1}$, the curve $v$, and the subsurface $S_g^1\subset S_{g+1}\setminus v$ as it is shown in Fig.~\ref{fig_s}. Further, we identify $S_{g,1}$ with $S_g^1\setminus \partial S_g^1$. Consider the rotation~$s$ by~$\pi$ about the horizontal axis in Fig.~\ref{fig_s}. (Such~$s$ is called a \textit{topological hyperelliptic involution}.) Then, the mapping class~$s$ belongs to~$\widehat{\I}_{g+1}\setminus\I_{g+1}$, stabilizes~$v$, and the restrictions of~$s$ to~$S_g^1$ and~$S_{g,1}$ belong to~$\widehat{\I}_g^1\setminus\I_g^1$ and~$\widehat{\I}_{g,1}\setminus\I_{g,1}$, respectively. So the involution~$s$ defines the required $\Z/2\Z$-actions on all groups in the sequence~\eqref{eq_mono_sequence}. It follows immediately that the last two homomorphisms in~\eqref{eq_mono_sequence} are $\Z/2\Z$-equivariant.
\begin{figure}
\begin{tikzpicture}[scale=.8]

\fill[violet!7] (1.5,-1.5) arc (270:90:0.3 and 1.5) -- (6,1.5) arc (90:-90:1.5) -- cycle
(3,0) circle (.4) (6,0) circle (.4);
\fill[violet!15] (1.5,-1.5) arc (-90:90:0.3 and 1.5) -- (6,1.5) arc (90:-90:1.5) -- cycle
(3,0) circle (.4) (6,0) circle (.4);

\draw[violet, very thick] (1.5,-1.5) arc (270:90:0.3 and 1.5);
\draw[violet, very thick, dashed] (1.5,-1.5) arc (-90:90:0.3 and 1.5);
\draw[violet] (4.5,-1) node
{$S_g^1$};

\draw[red, thick] (-1.5,0) arc (180:360:.55 and .15) node[pos=0.5,below] {$v$};
\draw[red, thick, dashed] (-.4,0) arc (0:180:.55 and .15);

\draw[black] (-3,0) -- (-1.8,0);
\draw[black] (7.8,0) -- (9,0);

\draw[black,->] (9.15,-.65) arc (-70:150:.2 and .7) node[pos=.3,right]{$s$};

\draw[black, ultra thick] (0,1.5) -- (6,1.5) arc (90:-90:1.5) -- (0,-1.5) arc (270:90:1.5);
\draw[black, ultra thick] (0,0) circle (.4);
\draw[black, ultra thick] (3,0) circle (.4);
\draw[black, ultra thick] (6,0) circle (.4);

\fill[black] (4.5,0) circle (1.5pt);
\fill[black] (4.8,0) circle (1.5pt);
\fill[black] (4.2,0) circle (1.5pt);

\end{tikzpicture}
 \caption{The hyperelliptic involution $s$}\label{fig_s}
\end{figure}

The first two isomorphisms in~\eqref{eq_mono_sequence} come from the Lyndon--Hochschild--Serre spectral sequences for the short exact sequences~\eqref{eq_ses} and~\eqref{eq_exact}, respectively. The hyperelliptic involution~$s$ defines the action of~$\Z/2\Z$ on these short exact sequences and hence on the corresponding Lyndon--Hochschild--Serre spectral sequences. Since the actions of~$s$ on~$\HH_2(\pi_1(S_g))=\HH_2(S_g)$ and on the group~$\Z$ in~\eqref{eq_exact} are trivial, we obtain that the isomorphisms in the sequence~\eqref{eq_mono_sequence} are  $\Z/2\Z$-equivariant.
\end{proof}

\begin{cor}\label{cor_invariant}
All Bestvina--Bux--Margalit homology classes in the groups~$\HH_{3g-5}(\I_g)$, $\HH_{3g-3}(\I_{g,1})$, $\HH_{3g-2}(\I_g^1)$, and~$\HH_{3g-2}\bigl(\Stab_{\I_{g+1}}(v)\bigr)$ are $\Z/2\Z$-invariant. So after the localization away from~$2$, they lie in the groups~$\HH_{3g-5}(\I_g;\Z[1/2])^+$, $\HH_{3g-3}(\I_{g,1};\Z[1/2])^+$, $\HH_{3g-2}(\I_g^1;\Z[1/2])^+$, and~$\HH_{3g-2}\bigl(\Stab_{\I_{g+1}}(v);\Z[1/2]\bigr)^+$ respectively.
\end{cor}

\begin{proof}
The group~$\I_2$ is generated by Dehn twists about separating curves. For each separating curve $\delta$ on~$S_2$, there exists an involution $s\in\Mod(S_2)$ that acts on $\HH_1(S_2)$ as~$-1$ and takes~$\delta$ to itself. It follows that the action of~$\Z/2\Z=\widehat{\I}_2/\I_2$ on~$\HH_1(\I_2)$ is trivial. So the required assertion follows from Proposition~\ref{propos_equivariant}.
\end{proof}

\subsection{Systems of linearly independent  classes in $\HH_{3g-5-n}(\I_g)$ for $1\le n\le g-3$.}
\label{subsection_A2}
We need the following generalization of the construction of abelian cycles.  
Suppose that~$K$ is a subgroup of a group~$G$. Let  $h_1,\ldots,h_k$ be pairwise commuting elements of~$G$ belonging to the centralizer of~$K$ and $\xi\in \HH_l(K)$ be an arbitrary homology class. Consider the homomorphism $\chi\colon \Z^k\times K\to G$ that sends  the generator of the $i$\textsuperscript{th} factor~$\Z$ to~$h_i$ for every~$i$ and is the identity on~$K$. We put
$$
\CA(h_1,\ldots,h_k;\xi)=\chi_*(\mu_k\times\xi)\in \HH_{k+l}(G),
$$
where $\mu_k$ is the standard generator of the group $\HH_k(\Z^k)\cong \Z$.

Suppose that $1\le n\le g-3$. Let $\CS_n$ be the set of all $(n+1)$-tuples $\CU=(U_0,\ldots,U_n)$ such that
\begin{itemize}
\item $\HH_1(S_g)=U_0\oplus\cdots\oplus U_n$,
\item the summands~$U_i$ are pairwise orthogonal with respect to the intersection form,
\item $\rank U_i=2$ for $i=0,\ldots,n-1$, and $\rank U_n=2g-2n$,
\item all summands in the decomposition
\begin{equation}\label{eq_x_decompose2}
x=x_0+\cdots+x_n,\qquad x_i\in U_i,
\end{equation}
are nonzero.
\end{itemize}

The following analog of Proposition~\ref{propos_infinite1} is straightforward.

\begin{propos}\label{propos_infinite2}
The set~$\CS_n$ is infinite.
\end{propos}

Now, we are going to give analogs of definitions and results from Subsection~\ref{subsection_A} for a splitting $\CU\in\CS_n$.

Every summand in the decomposition~\eqref{eq_x_decompose2} can be uniquely written as $x_i=l_ia_i$, where $a_i$ is primitive and~$l_i>0$. Unlike the case in Subsection~\ref{subsection_A}, now $a_0,\ldots,a_n$ is not a basis of a Lagrangian subspace of~$\HH_1(S_g)$ but only a part of such a basis. We denote by~$\fB(\CU)$ the set of all $n$-tuples $\bb=(b_1,\ldots,b_n)$ such that $b_i\in U_i$ and $a_i\cdot b_i=1$ for every $i=1,\ldots,n$. Unlike the case in Subsection~\ref{subsection_A}, we need to choose an element~$b_n$ in the last subgroup~$U_n$ but still do not need to choose~$b_0$  in $U_0$.

\begin{definition}
A \textit{$\delta$-multicurve} for~$\CU\in\CS_n$ is an $n$-component multicurve 
$$
\Delta=\delta_1\cup\cdots\cup\delta_{n}
$$
that satisfies the following conditions:
 \begin{enumerate}
 \item $\delta_1,\ldots,\delta_n$ are separating simple closed curves,
 \item $S_g\setminus\Delta$ consists of $n+1$ connected components $X_0,\ldots,X_{n}$ such that
\begin{itemize} 
\item $X_0$ is a once-punctured torus adjacent to~$\delta_1$,
 \item for every $i=1,\ldots,n-1$, the component~$X_i$ is a twice-punctured torus adjacent to~$\delta_{i}$ and~$\delta_{i+1}$,
 \item $X_n$ is a once-punctured oriented surface of genus~$g-n$ adjacent to~$\delta_n$,
 \end{itemize}
 \item the image of the homomorphism $\HH_1(X_i)\to \HH_1(S_g)$ induced by the inclusion is~$U_i$.
 \end{enumerate}
\end{definition}

\begin{definition}
A \textit{$\beta$-multicurve} for $\CU\in\CS_n$ compatible with a $\delta$-multicurve~$\Delta$ is a $2n$-component oriented multicurve 
$$
B=\beta_1\cup\beta_1'\cup\cdots\cup\beta_{n}\cup\beta_{n}'
$$
that satisfies the following conditions:
\begin{enumerate}
\item $\beta_i$ and $\beta_i'$ are homologous to each other oriented simple closed curves contained in~$X_i$,
\item the $n$-tuple $\bb=\bigl([\beta_1],\ldots,[\beta_{n}]\bigr)$ belongs to~$\fB(\CU)$,
\item the connected component of~$X_n\setminus(\beta_n\cup\beta_n')$ that is adjacent to~$\delta_n$ is a three-punctured sphere,
\end{enumerate}
see Fig.~\ref{fig_beta_delta2}. The union $\Gamma=B\cup\Delta$ will be called a \textit{$\beta\delta$-multicurve} for~$\CU$.
\end{definition}
 \begin{figure}
\begin{tikzpicture}[scale=.8]

\draw[mygreen, thick] (1.5,-1.5) arc (270:90:0.3 and 1.5)
node[pos=0,below]{$\delta_1$};
\draw[mygreen, thick, dashed] (1.5,-1.5) arc (-90:90:0.3 and 1.5);

\draw[mygreen, thick] (4.5,-1.5) arc (270:90:0.3 and 1.5)
node[pos=0,below]{$\delta_2$};
\draw[mygreen, thick, dashed] (4.5,-1.5) arc (-90:90:0.3 and 1.5);

\draw[mygreen, thick] (7.5,-1.5) arc (270:90:0.3 and 1.5)
node[pos=0,below]{$\delta_n$};
\draw[mygreen, thick, dashed] (7.5,-1.5) arc (-90:90:0.3 and 1.5);

\fill[violet!7] (10.5,-1.5) arc (270:90:0.3 and 1.5) -- (15,1.5) arc (90:-90:1.5) -- cycle
(12,0) circle (.4) (15,0) circle (.4);
\fill[violet!15] (10.5,-1.5) arc (-90:90:0.3 and 1.5) -- (15,1.5) arc (90:-90:1.5) -- cycle
(12,0) circle (.4) (15,0) circle (.4);

\draw[violet, very thick] (10.5,-1.5) arc (270:90:0.3 and 1.5)
node[pos=0,below]{$\varepsilon$};
\draw[violet, very thick, dashed] (10.5,-1.5) arc (-90:90:0.3 and 1.5);

\draw[mygreen, thick, snake=brace] (-1.4,1.75) -- (1.4,1.75) 
node[pos=.5,above]{$X_0$};
\draw[mygreen, thick, snake=brace] (1.6,1.75) -- (4.4,1.75)
node[pos=.5,above]{$X_1$};
\draw[violet, thick, snake=brace] (7.6,1.75) -- (10.4,1.75)
node[pos=.5,above]{$Z$};
\draw[violet, thick, snake=brace] (10.6,1.75) -- (16.4,1.75)
node[pos=.5,above]{$\psi\bigl(S_{g'}^1\bigr)$};
\draw[mygreen, thick, snake=brace] (7.6,2.75) -- (16.4,2.75)
node[pos=.5,above]{$X_n$};

\draw[myblue, thick] (3,1.5) arc (90:270:.15 and .55); 
\draw[myblue, thick,->] (3,1.5) arc (90:185:.15 and .55); 
\draw[myblue, thick, dashed] (3,1.5) arc (90:-90:.15 and .55)
node[pos=0.5,right]{\small{${}\!\beta_{1}$}}; 

\draw[myblue, thick] (3,-.4) arc (90:270:.15 and .55); 
\draw[myblue, thick,->] (3,-1.5) arc (270:175:.15 and .55); 
\draw[myblue, thick, dashed] (3,-.4) arc (90:-90:.15 and .55)
node[pos=0.5,right]{\small{${}\!\beta_{1}'$}}; 

\draw[myblue, thick] (9,1.5) arc (90:270:.15 and .55); 
\draw[myblue, thick,->] (9,1.5) arc (90:185:.15 and .55); 
\draw[myblue, thick, dashed] (9,1.5) arc (90:-90:.15 and .55)
node[pos=0.5,right]{\small{${}\!\beta_{n}$}}; 

\draw[myblue, thick] (9,-.4) arc (90:270:.15 and .55); 
\draw[myblue, thick,->] (9,-1.5) arc (270:175:.15 and .55); 
\draw[myblue, thick, dashed] (9,-.4) arc (90:-90:.15 and .55)
node[pos=0.5,right]{\small{${}\!\beta_{n}'$}};

\draw[black, ultra thick] (0,1.5) -- (15,1.5) arc (90:-90:1.5) -- (0,-1.5) arc (270:90:1.5);
\draw[black, ultra thick] (0,0) circle (.4);
\draw[black, ultra thick] (3,0) circle (.4);
\draw[black, ultra thick] (9,0) circle (.4);
\draw[black, ultra thick] (12,0) circle (.4);
\draw[black, ultra thick] (15,0) circle (.4);

\fill[black] (6,0) circle (1.5pt);
\fill[black] (6.3,0) circle (1.5pt);
\fill[black] (5.7,0) circle (1.5pt);
\fill[black] (13.5,0) circle (1.5pt);
\fill[black] (13.8,0) circle (1.5pt);
\fill[black] (13.2,0) circle (1.5pt);

\end{tikzpicture}
\caption{A $\beta\delta$-multicurve and a compatible embedding~$\psi$ for $\CU\in\CS_n$}\label{fig_beta_delta2}
\end{figure}

We use the same convention on which of the two components of~$\Delta$ in~$X_i$ is~$\beta_i$ and which is~$\beta'_i$, namely,
 the connected component of~$X_i\setminus(\beta_i\cup\beta_i')$ that is adjacent to~$\delta_i$ lies on the right-hand side from~$\beta_i$ and on the left-hand side from~$\beta'_i$.

We put $g'=g-n-1$; then $g'\ge 2$. Consider an oriented compact surface~$S_{g'}^1$
of genus~$g'$ with one boundary component. 

\begin{definition}
An orientation preserving embedding $\psi\colon S_{g'}^1\hookrightarrow S_g$ is said to be \textit{compatible} with a $\beta\delta$-multicurve~$\Gamma$ for~$\CU$ if it satisfies the following conditions:
\begin{enumerate}
\item the image of~$\psi$ is disjoint from~$\Gamma$ (and hence is contained in~$X_n$),
\item the homology class~$a_n$ can be represented by a simple closed curve lying in~$Z=X_n\setminus\psi\bigl(S_{g'}^1\bigr)$.
\end{enumerate}
\end{definition}

Note that the latter condition in this definition is equivalent to the requirement that $a_n$, $[\beta_n]$ is a basis of the image of the homomorphism $\HH_1(Z)\to \HH_1(S_g)$ induced by the inclusion.

As in Subsection~\ref{subsection_A}, we easily see that
\begin{itemize}
\item for each $\CU\in\CS_n$, there exists a $\delta$-multicurve~$\Delta$ for~$\CU$,
\item for each pair $(\CU,\Delta)$ there exists a $\beta$-multicurve~$B$ for~$\CU$ compatible with~$\Delta$ and with any prescribed set of homology classes of components $\bb\in\fB(\CU)$,
\item for each $\CU\in\CS_n$ and each $\beta\delta$-multicurve~$\Gamma$ for~$\CU$, there exists a compatible embedding~$\psi\colon S_{g'}^1\hookrightarrow S_g$.
\end{itemize}
To prove the last assertion note that we can first choose a simple closed curve~$\alpha_n\subset X_n$ in homology class~$a_n$ such that~$\alpha_n$ has geometric intersection number~$1$ with each of the curves~$\beta_n$ and~$\beta_n'$, and then choose a subsurface $\psi\bigl(S_{g'}^1\bigr)\subset X_n$ disjoint from~$\alpha_n$, $\beta_n$ and~$\beta'_n$.

We denote by~$\mathfrak{P}(\CU)$ the set of all pairs~$(\Gamma,\psi)$ such that $\Gamma$ is a $\beta\delta$-multicurve for~$\CU$ and $\psi\colon S_{g'}^1\hookrightarrow S_g$ is an embedding compatible with~$\Gamma$.

The embedding $\psi$ induces an injection $\psi_*\colon\I_{g'}^1\to\I_g$ by extending homeomorphisms by the identity. The mapping classes $T_{\beta_i}T_{\beta_i'}^{-1}$ and~$T_{\delta_i}$, where $i=1,\ldots,n$, pairwise commute and belong to the centralizer of the image of~$\psi_*$. Therefore, for each homology class $\xi\in \HH_{3g'-2}\bigl(\I_{g'}^1\bigr)$, we obtain a well-defined homology class
\begin{equation*}
A_{\Gamma,\psi,\xi}=\CA\bigl(
T_{\beta_1}T_{\beta_1'}^{-1},\ldots,T_{\beta_{n}}T_{\beta_{n}'}^{-1},
T_{\delta_1},\ldots, T_{\delta_{n}};\psi_*(\xi)
\bigr)\in \HH_{3g-5-n}(\I_g).
\end{equation*}

By Theorem~\ref{theorem_BBM2} we have that $\cd\bigl(\I_{g'}^1\bigr)=3g'-2$ and the top homology group $\HH_{3g'-2}(\I_{g',1})$ contains a free abelian subgroup of infinite rank.

\begin{theorem}\label{theorem_explicit}
Suppose that $1\le n\le g-3$. For each $\CU\in\CS_n$, choose an arbitrary pair $(\Gamma_{\CU},\psi_{\CU})\in\mathfrak{P}(\CU)$. Let $\{\xi_{\lambda}\}_{\lambda\in\Lambda}$ be an infinite system of linearly independent  homology classes in~$\HH_{3g'-2}\bigl(\I_{g'}^1\bigr)$.
Then the homology classes~$A_{\Gamma_{\CU},\psi_{\CU},\xi_{\lambda}}$, where $\CU$ runs over~$\CS_n$ and $\lambda$ runs over~$\Lambda$, are linearly independent in $\HH_{3g-5-n}(\I_g)$.
\end{theorem}

Since the sets $\CS_n$ and~$\Lambda$ are infinite,  Theorem~\ref{theorem_main} follows from Theorem~\ref{theorem_explicit}. (In fact, Theorem~\ref{theorem_main} would follow even if only one of the two sets $\CS_n$ and~$\Lambda$ were infinite.) We will prove Theorem~\ref{theorem_explicit} in Section~\ref{section_main_proof}.

\begin{propos}\label{propos_pm_2}
Suppose that $\xi\in \HH_{3g'-2}\bigl(\I_{g'}^1\bigr)$ is a Bestvina--Bux--Margalit homology class. Then
the class $A_{\Gamma,\psi,\xi}$ localized away from~$2$ lies in $\HH_{3g-5-n}(\I_g;\Z[1/2])^-$ if $n$ is odd and  lies in $\HH_{3g-5-n}(\I_g;\Z[1/2])^+$ if $n$ is even.
\end{propos}

\begin{proof}
Let $s$ be the rotation of~$S_g$  around the horizontal axis by~$\pi$ that stabilizes every curve~$\delta_i$ and the subsurface~$\psi\bigl(S_{g'}^1\bigr)$ and swaps the curves in every pair~$\{\beta_i,\beta_{i}'\}$ (cf. Fig.~\ref{fig_s} in Subsection~\ref{subsection_BBM}). Then $s$ acts on~$\HH_1(S_g)$ as~$-1$, hence, $s\in\widehat{\I}_g$ and $s\notin\I_g$. Therefore $\HH_{3g-5-n}(\I_g;\Z[1/2])^{\pm}$ are the eigenspaces with eigenvalues $\pm 1$ for the action of~$s$ on~$\HH_{3g-5-n}(\I_g;\Z[1/2])$. By Corollary~\ref{cor_invariant}, we have $s_*(\xi)=\xi$.
Consequently,
\begin{equation*}
s_*(A_{\Gamma,\psi,\xi})=\CA\bigl(
T_{\beta_1}^{-1}T_{\beta_1'},\ldots,T_{\beta_{n}}^{-1}T_{\beta_{n}'},
T_{\delta_1},\ldots, T_{\delta_{n}};\psi_*(\xi)\bigr)=(-1)^{n}A_{\Gamma,\psi,\xi},
\end{equation*}
which implies the required assertion.
\end{proof}

The assertion of Theorem~\ref{theorem_pm} for $k>2g-3$ follows immediately from Theorem~\ref{theorem_explicit} and Proposition~\ref{propos_pm_2}.

\section{Complex of cycles}\label{section_cycles}

In this section we recall in a convenient for us form some definitions and results of~\cite{BBM07}.

Let~$\cur$ be the set of isotopy classes of oriented non-separating simple closed curves on~$S_g$. Consider the cone~$\R_+^{\cur}$ consisting of all finite formal linear combinations $\sum_{\gamma\in\cur}l_{\gamma}\gamma$, where $l_{\gamma}$ are nonnegative real numbers. Endow~$\R_+^{\cur}$ with the standard direct limit topology. Notice that if $\gamma$ and $\widehat\gamma$ are two isotopy classes in~$\cur$ obtained from each other by reversing the orientation of the curve, then $\gamma$ and $\widehat\gamma$ are two different (linearly independent) basis vectors of~$\R_+^{\cur}$, so  $\widehat\gamma\ne -\gamma$.

Recall that we have fixed a primitive homology class~$x\in \HH_1(S_g)$.  A \textit{simple $1$-cycle} for~$x$ is  a finite linear combination $\sum_{i=1}^s l_i\gamma_i\in\R_+^{\cur}$, where $l_i>0$ and $\gamma_1,\ldots,\gamma_s$ are pairwise different isotopy classes in~$\cur$, that satisfies the following conditions:
\begin{enumerate}
\item the isotopy classes~$\gamma_1,\ldots,\gamma_s$ contain representatives such that the  union of these representatives is an oriented multicurve,
\item $\sum_{i=1}^s l_i[\gamma_i]=x$.
\end{enumerate} 
The oriented multicurve in condition~(1) (or its isotopy class) is called the \textit{support} of the simple $1$-cycle $\sum_{i=1}^s l_i\gamma_i$. A simple $1$-cycle $\sum_{i=1}^s l_i\gamma_i$ is called \textit{integral} if all coefficients~$l_i$ are integral.  A simple $1$-cycle $\sum_{i=1}^s l_i\gamma_i$ is called a
\textit{basic 1-cycle} if the homology classes $[\gamma_1],\ldots,[\gamma_s]$ are linearly independent. Notice that a basic $1$-cycle for the integral homology class $x$ is always integral.

Denote by~$\M$ the set of isotopy classes of oriented multicurves~$M$ on~$S_g$ satisfying the following conditions: 
\begin{enumerate}
\item For each component~$\gamma$ of~$M$, there exists a basic $1$-cycle for~$x$ whose support  is contained in~$M$ and contains~$\gamma$.

\item Any non-trivial linear combination of the homology classes of components of~$M$ with nonnegative coefficients is nonzero.
\end{enumerate}

For an oriented multicurve~$M\in\M$, let~$P_M$ be the convex hull in~$\R_+^{\cur}$ of all basic $1$-cycles for~$x$ with supports contained in~$M$. Since there are finitely many such basic $1$-cycles, $P_M$ is a finite-dimensional convex polytope.  It is easy to see that $P_M$ is exactly the set of all simple $1$-cycles for~$x$ with supports contained in~$M$. 
If $M,N\in\M$ and $N\subset M$, then $P_N$ is a face of~$P_M$. (Hereafter, the notation $N\subset M$ implies that the orientations of components of~$N$ are the same as their orientations in~$M$.) Vice versa, any face of~$P_M$ is~$P_N$ for certain oriented multicurve $N\in\M$ contained in~$M$.
Besides, it is not hard to check that the intersection of any two polytopes~$P_{M_1}$ and~$P_{M_2}$, where $M_1,M_2\in\M$, is either empty or a face of both of them. The union of all polytopes~$P_M$ with $M\in\M$ is a regular cell complex~$\B_g=\B_g(x)$, which was introduced by Bestvina, Bux, and Margalit~\cite{BBM07}  and was called by them the \textit{complex of cycles}. Denote by~$C_*(\B_g)$ the cellular chain complex of~$\B_g$ with integral coefficients.

\begin{remark}
Bestvina, Bux, and Margalit described the construction of~$\B_g$ in a slightly different way. Namely, they defined~$\B_g$ to be the quotient space of the disjoint union of the cells~$P_M$ with $M\in\M$ by identifying faces that are equal in $\R^{\cur}$ and endowing the quotient with the weak topology. These two constructions are equivalent to each other.
\end{remark}

For $M\in\M$, denote by $|M|$ the number of components of~$M$, by $c(M)$ the number of components of~$S_g\setminus M$, and by $D(M)$ the rank of the subgroup of $\HH_1(S_g)$ generated by the homology classes of components of~$M$. By~\cite[Lemma~2.1]{BBM07},
\begin{equation}\label{eq_dim}
\dim P_M=|M|-D(M)=c(M)-1.
\end{equation} 
We denote by~$\M_m$ the subset of~$\M$ consisting of all~$M$ such  that $\dim P_M=m$.

\begin{theorem}[Bestvina, Bux, Margalit~\cite{BBM07}]\label{theorem_BBM} The complex of cycles~$\B_g$ is contractible, provided that~$g\ge 2$.
\end{theorem}

\begin{remark}
In~\cite{BBM07} the authors were not careful enough about condition~(2) in the definition of~$\M$. Namely, they did not include explicitly this condition in the definition of the complex of cycles~$\B_g$ but later used it to prove that $\B_g$ is contractible. Possibly, they thought that condition~(2) is superfluous, since it follows from condition~(1). Nevertheless, the following example shows that this is not true. Let $a_1$, $a_2$, $a_3$ be three linearly independent homology classes that have trivial pairwise algebraic intersection numbers and generate a direct summand of $\HH_1(S_g)$. Suppose that  $x=a_1+a_2+2a_3$. Consider a $5$-component oriented multicurve~$M$ with components $\gamma_1$, $\gamma_2$, $\gamma_3$, $\gamma_4$, and~$\gamma_5$ whose homology classes are~$a_1$, $a_2$, $a_3$, $-a_1-a_2$, and~$a_1+a_2+a_3$, respectively. Then $\gamma_1+\gamma_2+2\gamma_3$ and $\gamma_4+2\gamma_5$ are basic $1$-cycles for~$x$. Therefore, $M$ satisfies condition~(1) and does not satisfy condition~(2). All results in~\cite{BBM07} become completely correct if one includes condition~(2) explicitly in the definition of the set~$\M$. Notice also that a thorough description of the complex of cycles including condition~(2) was given in~\cite{HaMa12}.
\end{remark}

The Torelli group~$\I_g$ acts cellularly and without rotations on the contractible complex~$\B_g$. `Without rotation' means that if an element $h\in\I_g$ stabilizes a cell~$P_M$, then it stabilizes every point of~$P_M$. The fact that the action of~$\I_g$ on~$\B_g$ is without rotations follows from a theorem by Ivanov~\cite[Theorem~1.2]{Iva92} saying that if an element $h\in\I_g$ stabilizes a multicurve~$M$, then it stabilizes every component of~$M$.

\section{Spectral sequence}\label{section_CL}

In this section we recall some standard facts on a spectral sequence associated with a cellular action of a group on a CW complex~$X$. In fact, there are two spectral sequences associated with such an action. In~\cite[Section~VII.7]{Bro82}, they are referred to as the \textit{first} and the \textit{second} spectral sequences. We are interested in the \textit{second} spectral sequence in terminology of~\cite{Bro82}. Usually, this spectral sequence is written for a cellular action of a group~$G$ on a CW complex~$X$. However, we need to consider the following more general purely algebraic situation.

Let $G$ be a group. Let $(C_*,\partial)$ be a chain complex of left $G$-modules satisfying the following conditions:
\begin{enumerate}
\item $C_k=0$ for $k<0$,
\item every $C_k$ is a free abelian group with a fixed basis~$\CE_k$, 
\item $G$ acts on~$C_k$ by permutations of elements of~$\CE_k$, i.\,e., $g\sigma\in\CE_k$ for all $g\in G$ and all $\sigma\in\CE_k$. 
\end{enumerate}

If $G$ acts cellularly and without rotations on a CW complex~$X$, then the cellular chain complex~$C_*(X)$ with the basis consisting of cells of~$X$ with appropriate orientations satisfies these condition. The orientations of cells should agree each other in every $G$-orbit, which is possible, since the action is without rotations.

Recall that the tensor product $A\otimes_GB$ of two left $G$-modules~$A$ and~$B$ is the quotient of the abelian group~$A\otimes B$ by all relations of the form $(ga)\otimes (gb)=a\otimes b$, where $g\in G$.

Let $(R_*,d)$ be a projective resolution of~$\Z$ over the group ring~$\Z G$, where $\Z$ is endowed with the structure of left $\Z G$-module so that $g\cdot 1=1$ for all $g\in G$. Consider the double complex
$$
B_{p,q}=C_p\otimes_GR_q,
$$
with anti-commuting differentials~$\partial'$ and~$\partial''$ of bi-degrees~$(-1,0)$ and~$(0,-1)$, respectively, given by
$$
\partial'(z\otimes\kappa)=(\partial z)\otimes\kappa,\qquad
\partial''(z\otimes\kappa)=(-1)^{\deg z}z\otimes(d\kappa).
$$ 

The spectral sequence we are interested in is the \textit{second} spectral sequence~$E^*_{*,*}$ of the double complex~$B_{*,*}$ in terminology of~\cite[Section~VII.3]{Bro82}. This is a first-quadrant spectral sequence with differentials~$d^r$ of bi-degrees $(-r,r-1)$, respectively, such that  $E^0_{p,q}=B_{p,q}$, $d^0=\partial''$, and $d^1$ is induced by~$\partial'$. We will need the following standard facts on this spectral sequence.

\begin{fact}\label{fact_E1}
Let $\Sigma_p\subseteq \CE_p$ be a subset containing exactly one representative in each $G$-orbit. Then there is a canonical isomorphism 
\begin{equation}\label{eq_E1pq}
E^1_{p,q}\cong\bigoplus_{\sigma\in \Sigma_p}\HH_q(\Stab_{G}(\sigma)).
\end{equation}
In particular, starting from the $E^1$-page, the spectral sequence is independent of the choice of a projective resolution~$R_*$.
Besides, if we choose another set of representatives of $G$-orbits~$\Sigma'_p=\{g_{\sigma}\sigma\}_{\sigma\in\Sigma_p}$, where $g_{\sigma}\in G$, then the two corresponding isomorphisms~\eqref{eq_E1pq} will differ by the direct sum of the isomorphisms $\HH_q(\Stab_{G}(\sigma))\to \HH_q(\Stab_{G}(g_{\sigma}\sigma))$ induced by the conjugations by~$g_{\sigma}$.
 \end{fact}
 
We denote the canonical injection $\HH_q(\Stab_{G}(\sigma))\hookrightarrow E^1_{p,q}$ by~$\iota_{\sigma}$.
 
\begin{remark}\label{remark_sign}
If $C_*=C_*(X)$, then the basis consisting of the cells of~$X$ is canonically defined only up to signs. More precisely, we may reverse simultaneously the orientations of all cells in any $G$-orbit. It is easy to see that the injection~$\iota_{\sigma}$ reverses its sign whenever we reverse the orientation of~$\sigma$.
\end{remark}

 \begin{fact}\label{fact_Einfty} $\bigoplus_{p+q=s}E^{\infty}_{p,q}$ is the graded group associated with certain filtration
$$
0=\CF_{-1,s+1}\subseteq \CF_{0,\,s}\subseteq \CF_{1,\,s-1}\subseteq \cdots \subseteq \CF_{s,\,0}=\HH_s^G(C_*)
$$
so that
$$E^{\infty}_{p,q}=\CF_{p,q}/\CF_{p-1,\,q+1}.$$
Here~$\HH_s^G(C_*)$ is the $s$\textsuperscript{th} $G$-equivariant homology group of~$C_*$, i.\,e., the $s$\textsuperscript{th}  homology group of the total complex of the double complex~$B_{*,*}$.  Moreover, $\CF_{p,q}$ is exactly the image of the natural homomorphism
$$
\HH_{p+q}^G(C_*^{\le p})\to \HH_{p+q}^G(C_*),
$$
where $C_i^{\le p}=C_i$ for $i\le p$ and $C_i^{\le p}=0$ for $i>p$. Recall that if $C_*=C_*(X)$ and $X$ is contractible, then~$\HH_s^G(C_*)$ is naturally isomorphic to~$\HH_s(G)$.
 \end{fact}

In the `topological case' $C_*=C_*(X)$, the proofs of Facts~\ref{fact_E1} and~\ref{fact_Einfty} can be found in~\cite[Section~VII.7]{Bro82}, and~\eqref{eq_E1pq} is exactly isomorphism~(7.7) in~\cite{Bro82}. The proofs for an arbitrary chain complex~$C_*$ are literally the same.
Obviously, the construction of the spectral sequence we consider is functorial in the following sense.

\begin{fact}\label{fact_natural}
For $i=1,2$, let $C_*^{(i)}$ be a chain complex of left $G_i$-modules satisfying \mbox{\textnormal{(1)--(3)}} and let $E^*_{*,*}(i)$  be the corresponding spectral sequence. Suppose that $\chi\colon G_1\to G_2$ is a homomorphism and $\varphi\colon C_*^{(1)}\to C_*^{(2)}$ is a $\chi$-equivariant chain map. Then the pair~$(\varphi,\chi)$ induces a morphism of spectral sequences $E^*_{*,*}(1)\to E^*_{*,*}(2)$ that in the $E^{\infty}$-page is associated with the natural homomorphism $\HH_*^{G_1}\bigl(C_*^{(1)}\bigr)\to \HH_*^{G_2}\bigl(C_*^{(2)}\bigr)$.
\end{fact}

Recall that if $A$ is a left $H$-module and $H$ is a subgroup of~$G$, then  the left $G$-module
$$
\Ind_H^GA=\Z G\otimes_HA
$$
is called the \textit{induced module}. Here~$\otimes_H$ denotes the tensor product over~$\Z H$ of~$\Z G$ regarded as a right $H$-module with the left $H$-module~$A$. The structure of the left $G$-module on~$\Ind_H^GA$ is induced by the structure of the left $G$-module on~$\Z G$, see~\cite[Section~III.5]{Bro82}.

\begin{fact}\label{fact_ind}
Suppose that $C_*$ is a chain complex of left $H$-modules satisfying~\mbox{\textnormal{(1)--(3)}} and $H$ is a subgroup of~$G$. Let $E^*_{*,*}(1)$ be the spectral sequence
for the action of~$H$ on~$C_*$ and let $E^*_{*,*}(2)$ be the spectral sequence
for the action of~$G$ on the chain complex $D_*=\Ind_H^GC_*$. Then the spectral sequences~$E^*_{*,*}(1)$ and~$E^*_{*,*}(2)$ are naturally isomorphic to each other starting from the $E^1$-page.
\end{fact}

\begin{proof}
By Fact~\ref{fact_natural}, the injection $H\subset G$ and the injection $\varphi\colon C_*\to D_*$ given by~$\xi\mapsto 1\otimes\xi$ induce a well-defined morphism of spectral sequences $\Phi\colon E^*_{*,*}(1)\to E^*_{*,*}(2)$.  Suppose that  $\Sigma_p$ is a set of representatives of $H$-orbits of basis elements in~$C_p$. Then the elements $1\otimes\sigma$, where $\sigma\in\Sigma_p$, constitute the set of representatives of $G$-orbits of basis elements in~$D_p$. The morphism~$\Phi$ in the $E^1$-page coincides with the isomorphism
$$
E^1_{p,q}(1)\cong \bigoplus_{\sigma\in \Sigma_p}\HH_q(\Stab_{H}(\sigma))=
\bigoplus_{\sigma\in \Sigma_p}\HH_q(\Stab_{G}(1\otimes\sigma))\cong E^1_{p,q}(2).
$$
Therefore $\Phi$ is an isomorphism starting from the $E^1$-page.
\end{proof}

\section{The spectral sequence and generalized abelian cycles}\label{section_lemma}
Suppose that $X$ is a contractible regular cell complex and $G$ is a group acting on~$X$ cellularly and without rotations. Let $E^*_{*,*}$ be the  corresponding spectral sequence from Section~\ref{section_CL} and let~$\CF_{*,*}$ be the corresponding filtration in~$\HH_*(G)$.

Consider the cube $[0,1]^n$ with  standard coordinates $t_1,\ldots,t_n$. Let $F_i$ be the facet of~$[0,1]^n$ given by~$t_i=0$, and let $F_i'$ be the facet of~$[0,1]^n$ given by~$t_i=1$. We denote by~$\tau_i$ the standard homeomorphism $F_i\to F_i'$ given by the parallel translation.

\begin{lem}\label{propos_spectral}
Let $f_1,\dots, f_n$ be pairwise commuting elements of~$G$ and let $P$ be an $n$-cell of~$X$ satisfying the following conditions:
\begin{itemize}
\item there is a homeomorphism $\varphi\colon [0,1]^n\to P$ that maps every face of~$[0,1]^n$ homeomorphically onto a subcell of~$P$,
\item for every $i=1,\ldots,n$, the action by~$f_i$ maps the cell~$\varphi(F_i)$ onto the cell~$\varphi(F'_i)$ via the homeomorphism $\varphi\circ\tau_i\circ\varphi^{-1}$,
\item $f_1,\dots,f_n$ lie in the centralizer of the stabilizer~$\Stab_{G}(P)$.
\end{itemize}
Suppose that $u\in \HH_q(\Stab_{G}(P))$. Then the homology class~$\CA(f_1,\ldots,f_n;u)$ lies in the subgroup $\CF_{n,q}\subseteq \HH_{n+q}(G)$ and the image of this class under the projection
$$
\CF_{n,q}\to\CF_{n,q}/\CF_{n-1,\,q+1}=E^{\infty}_{n,q}
$$
is represented by the class $\iota_{P}(u)\in E^{1}_{n,q}$.
\end{lem}

\begin{proof}
We put $H=\Stab_{G}(P)$. 

Let $Y$ be the standard decomposition of~$\R^n$ into unit cubes with vertices at points with integral coordinates. In particular, the cube $Q=[0,1]^n$ is a cell of~$Y$. The group~$\Z^n$ acts on~$Y$ by translations. This action is cellular and free. Let $C_*(Y/\Z^n)$ be the cellular chain complex of the cell complex~$Y/\Z^n$, which is the standard cell decomposition of the $n$-dimensional torus $\R^n/\Z^n$. It is easy to see that the differential of~$C_*(Y/\Z^n)$ is trivial, hence, $C_*(Y/\Z^n)=\HH_*(Y/\Z^n)=\HH_*(\Z^n)$.

Let $\uE^*_{*,*}$ be the spectral sequence from Section~\ref{section_CL} for the action of~$\Z^n\times H$ on~$Y$, where the factor~$H$ acts  trivially, and let $\uF_{*,*}$ be the corresponding filtration in~$\HH_*(\Z^n\times H)$. Obviously, $\uE^r_{p,q}=0$ off the strip $0\le p\le n$. Since the stabilizer of every face of~$Y$ is exactly~$H$, isomorphism~\eqref{eq_E1pq} takes the form
\begin{equation}\label{eq_iso1}
\uE^1_{n,q}\cong C_n(Y/\Z^n)\otimes \HH_q(H)=\HH_n(\Z^n)\otimes \HH_q(H).
\end{equation}
It follows from Fact~\ref{fact_Einfty} and the K\"unneth theorem that
\begin{gather*}
\uF_{n,q}=\HH_{n+q}(\Z^n\times H)=\bigoplus_{i+j=n+q}\HH_i(\Z^n)\otimes \HH_j(H),\\
\uF_{n-1,q+1}=\bigoplus_{i+j=n+q,\,i<n}\HH_i(\Z^n)\otimes \HH_j(H).
\end{gather*}
Hence we obtain a natural isomorphism
\begin{equation}\label{eq_isoinfty}
\uE^{\infty}_{n,q}\cong \HH_n(\Z^n)\otimes \HH_q(H).
\end{equation}
It can be easily checked that the isomorphism~\eqref{eq_isoinfty} is the composition of the natural inclusion $\uE^{\infty}_{n,q}\subseteq \uE^1_{n,q}$ and the isomorphism~\eqref{eq_iso1}. In particular this implies that $\uE^{\infty}_{n,q}=\uE^1_{n,q}$. (In fact, it is not hard to prove that the spectral sequence~$\uE^*_{*,*}$ stabilizes at the $E^1$-page.) Since the projection $\HH_{n+q}(\Z^n\times H)\to \HH_n(\Z^n)\otimes \HH_q(H)$ takes $\mu_n\times u$ to $\mu_n\otimes u$ and isomorphism~\eqref{eq_iso1} takes $\iota_Q(u)$ to $\mu_n\otimes u$, we obtain that the projection  $\uF_{n,q}\to \uE^{\infty}_{n,q}=\uE^1_{n,q}$ takes $\mu_n\times u$ to~$\iota_Q(u)$.

Now, let $\chi\colon\Z^n\times H\to G$ be the homomorphism that takes the elements of the standard basis of~$\Z^n$ to $f_1,\ldots,f_n$, respectively, and is the identity on~$H$.
Then the homeomorphism $\varphi\colon Q\to P$ has a unique $\chi$-equivariant  extension $\widetilde{\varphi}\colon Y\to X$. By Fact~\ref{fact_natural}, the pair~$(\widetilde{\varphi},\chi)$ induces a morphism of spectral sequences $\Phi\colon\uE^*_{*,*}\to E^*_{*,*}$ that in the $E^{\infty}$-page is associated with the homomorphism $\chi_*\colon \HH_*(\Z^m\times H)\to \HH_*(G)$. Hence $\chi_*(\uF_{i,j})\subseteq \CF_{i,j}$ for all~$i$ and~$j$.
Therefore the class $\CA(f_1,\ldots,f_n;u)=\chi_*(\mu_n\times u)$ belongs to~$\CF_{n,q}$ and 
projects onto the class in~$E^{\infty}_{n,q}$ represented by the element $\iota_{P}(u)=\Phi(\iota_{Q}(u))\in E^{1}_{n,q}$. 
\end{proof}

\section{Auxiliary spectral sequences~$E^*_{*,*}(\fN)$ and~$\hE^*_{*,*}(\fN)$}\label{section_aux_spectral}
Consider the complex of cycles~$\B_g$ and the cellular chain complex $C_*=C_*(\B_g)$. Throughout the rest of the present paper  $E^*_{*,*}$ is the spectral sequence from Section~\ref{section_CL} for the action of~$\I_g$ on~$\B_g$ and $\CF_{*,*}$ is the corresponding filtration in~$\HH_*(\I_g)$.

The Torelli group~$\I_g$ acts naturally on~$\M_n$ for each~$n$.
\begin{definition}
We say that an $\I_g$-orbit $\fN\subseteq\M_n$ is \textit{perfect} if no oriented multicurve $M\in\M$ contains two distinct submulticurves~$N_1$ and~$N_2$  belonging to~$\fN$.
\end{definition}

For each~$m$, the basis of~$C_m$ consists of all~$P_M$, where $M\in\M_m$. Consider an arbitrary $\I_g$-orbit $\fN\subseteq\M_n$. Let $K_m(\mathfrak{N})\subseteq C_m$ be the subgroup generated by all~$P_M$ such that $M$ does not contain an oriented submulticurve belonging to~$\mathfrak{N}$. Obviously, $K_m(\mathfrak{N})$ is invariant under the action of~$\I_g$ and $K_*(\mathfrak{N})$ is a chain subcomplex of~$C_*$. Consider the quotient chain complex of $\I_g$-modules 
$$
C_*^{\mathfrak{N}}=C_*/K_*(\mathfrak{N}),
$$
and denote its differential by~$\partial_{\mathfrak{N}}$. 

The group~$C_m^{\fN}$ is free abelian  with the basis consisting of all~$P_M$ such that $M\in\M_m$ and $M$ contains a submulticurve belonging to~$\fN$. (With some abuse of notation, we denote the image of~$P_M$ under the projection $C_m\to C_m^{\fN}$ again by~$P_M$.)  The group~$\I_g$ acts on~$C_m^{\fN}$ by permutations of the elements of this basis. Therefore 
the spectral sequence from Section~\ref{section_CL} for the action of~$\I_g$ on~$C_*^{\fN}$ is well defined. We denote this spectral sequence by~$E^*_{*,*}(\fN)$. By Fact~\ref{fact_natural}, the projection $C_*\to C_*^{\fN}$ induces a morphism of spectral sequences
$$
\Pi_{\fN}\colon E^*_{*,*}\to E^*_{*,*}(\fN).
$$

Now, choose an oriented multicurve $N\in\fN$. Let $C_m^N\subseteq C_m^{\fN}$ be the subgroup generated by all~$P_M$ such that  $M\in\M_m$ and $M$ contains~$N$. 

\begin{propos}
Suppose that $\fN$ is perfect. Then $C_*^N$ is a chain subcomplex of~$C_*^{\fN}$ and
$$
C_m^{\fN}=\Ind_{\Stab_{\I_g}(N)}^{\I_g}C_m^N
$$
for all~$m$.
\end{propos}

\begin{proof}
To prove that $C_*^N$ is a chain subcomplex of~$C_*^{\fN}$, we need to show that $\partial_{\fN}P_M\in C_{m-1}^N$ whenever $M\in\M_m$ and $M\supseteq N$. We immediately see from the construction of~$C_*^{\fN}$ that  $\partial_{\fN}P_M$ is an algebraic sum  of all~$P_{L}$ such that $L\in\M_{m-1}$, $L\subseteq M$, and $L$ contains an oriented multicurve in the orbit~$\fN$. However, since $\fN$ is perfect, $M$ does not contain a multicurve in~$\fN$ different from~$N$. Hence $L\supseteq N$ for all~$P_L$ entering~$\partial_{\fN}P_M$. Therefore $ \partial_{\fN}P_M\in C_{m-1}^N$.

Let $M_1, M_2, \ldots$ be representatives of all different $\Stab_{\I_g}(N)$-orbits of oriented multicurves $M\in\M_m$ containing~$N$. Then $C_m^N=\bigoplus A_i$, where $A_i$ is a cyclic $\Stab_{\I_g}(N)$-module with generator~$M_i$. 
Let us prove that no two different multicurves~$M_i$ and~$M_j$ belong to the same $\I_g$-orbit. Indeed, suppose that $M_j=f(M_i)$, where $f\in\I_g$. Then $f(N)$ is an oriented multicurve in~$\fN$ that is contained in~$M_j$. Since $\fN$ is perfect, we obtain that $f(N)=N$, hence, $f\in\Stab_{\I_g}(N)$, which is impossible, since $M_i$ and~$M_j$ are representatives of different  $\Stab_{\I_g}(N)$-orbits. 

The basis of~$C_m^{\fN}$ consists of all~$P_{M}$ such that $M\in\M_m$ and $M$ contains a multicurve belonging to~$\fN$. Any such~$M$ can be written as $f(M_i)$ for certain~$i$ and certain~$f\in\I_g$. Moreover, the multicurve~$M_i$ is uniquely determined by~$M$, since no two different multicurves~$M_i$ and~$M_j$ belong to the same $\I_g$-orbit. Hence $C_m^{\fN}=\bigoplus B_i$, where $B_i$ is a cyclic $\I_g$-module with generator~$M_i$. Besides, for each~$i$, the multicurve~$N$ is the only multicurve in~$\fN$ that is contained in~$M_i$. Therefore $\Stab_{\I_g}(M_i)\subseteq \Stab_{\I_g}(N)$. It easily follows that $B_i=\Ind_{\Stab_{\I_g}(N)}^{\I_g}A_i$ for all~$i$. Thus, $C_m^{\fN}=\Ind_{\Stab_{\I_g}(N)}^{\I_g}C_m^N$.
\end{proof}

By Fact~\ref{fact_ind}, the spectral sequence~$E^*_{*,*}(\fN)$ for the action of~$\I_g$ on~$C_*^{\fN}$ coincides (starting from the $E^1$-page) with the spectral sequence for the action of~$\Stab_{\I_g}(N)$ on~$C_*^{N}$.

Now, let $\BP(N)$ be the subgroup of~$\I_g$ generated by all BP maps corresponding to bounding pairs of components of~$N$. Consider  the following variant of the Birman--Lubotzky--McCarthy exact sequence (cf.~\cite[Section~6.2]{BBM07}):
\begin{equation}\label{eq_BLMC}
1\to \BP(N)\to \Stab_{\I_g}(N)\xrightarrow{j_N}\PMod(S_g\setminus N).
\end{equation}
Here $\PMod(S_g\setminus N)$ is the \textit{pure mapping class group} of~$S_g\setminus N$, that is, the subgroup of~$\Mod(S_g\setminus N)$ consisting of all mapping classes that stabilize every puncture of~$S_g\setminus N$. Obviously, $\PMod(S_g\setminus N)$ is the direct product of the groups~$\PMod(Y_i)$ over all connected components~$Y_i$ of~$S_g\setminus N$. We put 
\begin{equation*}
H_N=j_N\bigl(\Stab_{\I_g}(N)\bigr)\cong \Stab_{\I_g}(N)/\BP(N).
\end{equation*}

Obviously, the action of~$\BP(N)$  on~$C_*^N$ is trivial. Therefore, $C_*^N$ is a chain complex of $H_N$-modules. Let $\hE^*_{*,*}(\fN)$ be the spectral sequence from Section~\ref{section_CL} for the action of~$H_N$ on~$C_*^N$. By Fact~\ref{fact_natural}, the homomorphism~$j_N$ induces a morphism of spectral sequences
$$
J_{\fN}\colon E^*_{*,*}(\fN)\to\hE^*_{*,*}(\fN).
$$
It is not hard to see that the spectral sequence~$\hE^*_{*,*}(\fN)$ and the morphism~$J_{\fN}$ are independent (up to a canonical isomorphism) of the choice of a multicurve $N\in\fN$. Indeed, if $N'$ is another multicurve in~$\fN$, then for any mapping class $f\in\I_g$ taking $N$ to~$N'$, the conjugation by~$f$ yields a natural isomorphism from the spectral sequence for the action of~$H_N$ on~$C_*^N$ to the spectral sequence for the action of~$H_{N'}$ on~$C_*^{N'}$, and this isomorphism is independent of the choice of~$f$.

Finally, we obtain the morphism of spectral sequences
$$
\hPi_{\fN}=J_{\fN}\circ\Pi_{\fN}\colon E^*_{*,*}\to\hE^*_{*,*}(\fN).
$$

The canonical isomorphism~\eqref{eq_E1pq} takes the form
\begin{equation}\label{eq_hE1}
\hE^1_{p,q}(\fN)\cong\bigoplus_{M\in\M_p\colon M\supseteq N}\HH_q(\Stab_{H_N}(M)).
\end{equation}
In particular, if $N\in\M_n$, then
\begin{align}
\hE^1_{p,q}(\fN)&=0,\qquad p<n,\label{eq_hE1_N_zero}\\
\hE^1_{n,q}(\fN)&\cong \HH_q(H_N).\label{eq_hE1_N}
\end{align}
The latter isomorphism is inverse to~$\iota_{P_N}$. Hence, by Remark~\ref{remark_sign}, it changes sign whenever one reverses the orientation of~$P_N$.

\section{$\alpha$-multicurves}\label{section_alpha}

To prove Theorems~\ref{theorem_Abelian_explicit} and~\ref{theorem_explicit} we will apply the construction of auxiliary spectral sequences from the previous section to some special oriented multicurves~$N$, which will be called \textit{$\alpha$-multicurves}. We will introduce two types of $\alpha$-multicurves; type one $\alpha$-multicurves will be used in the proof of Theorem~\ref{theorem_Abelian_explicit} and type two $\alpha$-multicurves will be used in the proof of Theorem~\ref{theorem_explicit}.

\begin{definition}\label{defin_alpha}
A \textit{type one $\alpha$-multicurve} is a $(2g-2)$-component oriented multicurve
\begin{equation}\label{eq_N}
N=\alpha_0\cup\alpha_1\cup\alpha_1'\cup\cdots\cup\alpha_{g-2}\cup\alpha_{g-2}'\cup\alpha_{g-1}
\end{equation}
that satisfies the following conditions (see Fig.~\ref{fig_alpha}):
\begin{enumerate}
 \item all components of~$N$ are non-separating simple closed curves,
 \item $\alpha_i'$ is homologous to~$\alpha_i$ for $i=1,\ldots,g-2$,
 \item $S_g\setminus N$ consists of~$g-1$ connected components~$Y_1,\ldots,Y_{g-1}$  such that
 every $Y_i$ is a four-punctured sphere that is adjacent to~$\alpha_i$ and~$\alpha_{i-1}'$ from the left and is adjacent to~$\alpha_i'$ and~$\alpha_{i-1}$ from the right, where we use the convention $\alpha_0'=\alpha_0$ and $\alpha_{g-1}'=\alpha_{g-1}$.
 \item $x=l_0[\alpha_0]+\cdots+l_{g-1}[\alpha_{g-1}]$ for some positive integers $l_0,\ldots,l_{g-1}$.
\end{enumerate}
\end{definition}

\begin{figure}
\begin{tikzpicture}[scale=.8]

\draw[red, thick, snake=brace] (2.9,-1.75) -- (-1.4,-1.75) 
node[pos=.47,below]{$Y_1$};
\draw[red, thick, snake=brace] (5.9,-1.75) -- (3.1,-1.75)
node[pos=.47,below]{$Y_2$};
\draw[red, thick, snake=brace] (16.4,-1.75) -- (12.1,-1.75)
node[pos=.42,below]{$Y_{g-1}$};

\draw[red, thick] (3,1.5) arc (90:270:.15 and .55); 
\draw[red, thick,->] (3,.4) arc (270:175:.15 and .55); 
\draw[red, thick, dashed] (3,1.5) arc (90:-90:.15 and .55)
node[pos=0.5,right]{\small{${}\!\alpha_{1}$}}; 

\draw[red, thick] (3,-.4) arc (90:270:.15 and .55); 
\draw[red, thick,->] (3,-.4) arc (90:185:.15 and .55); 
\draw[red, thick, dashed] (3,-.4) arc (90:-90:.15 and .55)
node[pos=0.5,right]{\small{${}\!\alpha_{1}'$}}; 

\draw[red, thick] (6,1.5) arc (90:270:.15 and .55); 
\draw[red, thick,->] (6,.4) arc (270:175:.15 and .55); 
\draw[red, thick, dashed] (6,1.5) arc (90:-90:.15 and .55)
node[pos=0.5,right]{\small{${}\!\alpha_{2}$}}; 

\draw[red, thick] (6,-.4) arc (90:270:.15 and .55); 
\draw[red, thick,->] (6,-.4) arc (90:185:.15 and .55); 
\draw[red, thick, dashed] (6,-.4) arc (90:-90:.15 and .55)
node[pos=0.5,right]{\small{${}\!\alpha_{2}'$}}; 

\draw[red, thick] (12,1.5) arc (90:270:.15 and .55); 
\draw[red, thick,->] (12,.4) arc (270:175:.15 and .55); 
\draw[red, thick, dashed] (12,1.5) arc (90:-90:.15 and .55)
node[pos=0.5,right]{\small{${}\!\alpha_{g-2}$}}; 

\draw[red, thick] (12,-.4) arc (90:270:.15 and .55); 
\draw[red, thick,->] (12,-.4) arc (90:185:.15 and .55); 
\draw[red, thick, dashed] (12,-.4) arc (90:-90:.15 and .55)
node[pos=0.5,right]{\small{${}\!\alpha_{g-2}'$}}; 

\draw[red, thick] (-1.5,0) arc (180:360:.55 and .15)
node[pos=0.5,below]{\small{${}\!\alpha_0$}}; 
\draw[red, thick,->] (-.4,0) arc (0:-95:.55 and .15); 
\draw[red, thick, dashed] (-.4,0) arc (0:180:.55 and .15); 

\draw[red, thick] (15.4,0) arc (180:360:.55 and .15)
node[pos=0.5,below]{\small{${}\!\!\alpha_{g-1}$}}; 
\draw[red, thick,->] (15.4,0) arc (180:275:.55 and .15); 
\draw[red, thick, dashed] (16.5,0) arc (0:180:.55 and .15);

\draw[black, ultra thick] (0,1.5) -- (15,1.5) arc (90:-90:1.5) -- (0,-1.5) arc (270:90:1.5);
\draw[black, ultra thick] (0,0) circle (.4);
\draw[black, ultra thick] (3,0) circle (.4);
\draw[black, ultra thick] (6,0) circle (.4);
\draw[black, ultra thick] (12,0) circle (.4);
\draw[black, ultra thick] (15,0) circle (.4);

\fill[black] (9,0) circle (1.5pt);
\fill[black] (9.3,0) circle (1.5pt);
\fill[black] (8.7,0) circle (1.5pt);

\end{tikzpicture}
\caption{Type one $\alpha$-multicurve}\label{fig_alpha}
\bigskip
\bigskip
\bigskip
\begin{tikzpicture}[scale=.8]

\draw[red, thick, snake=brace] (2.4,-1.75) -- (-1.4,-1.75) 
node[pos=.47,below]{$Y_1$};
\draw[red, thick, snake=brace] (4.9,-1.75) -- (2.6,-1.75)
node[pos=.47,below]{$Y_2$};
\draw[red, thick, snake=brace] (16.4,-1.75) -- (9.6,-1.75)
node[pos=.46,below]{$Y_{n+1}$};

\draw[red, thick] (2.5,1.5) arc (90:270:.15 and .55); 
\draw[red, thick,->] (2.5,.4) arc (270:175:.15 and .55); 
\draw[red, thick, dashed] (2.5,1.5) arc (90:-90:.15 and .55)
node[pos=0.5,right]{\small{${}\!\alpha_{1}$}}; 

\draw[red, thick] (2.5,-.4) arc (90:270:.15 and .55); 
\draw[red, thick,->] (2.5,-.4) arc (90:185:.15 and .55); 
\draw[red, thick, dashed] (2.5,-.4) arc (90:-90:.15 and .55)
node[pos=0.5,right]{\small{${}\!\alpha_{1}'$}}; 

\draw[red, thick] (5,1.5) arc (90:270:.15 and .55); 
\draw[red, thick,->] (5,.4) arc (270:175:.15 and .55); 
\draw[red, thick, dashed] (5,1.5) arc (90:-90:.15 and .55)
node[pos=0.5,right]{\small{${}\!\alpha_{2}$}}; 

\draw[red, thick] (5,-.4) arc (90:270:.15 and .55); 
\draw[red, thick,->] (5,-.4) arc (90:185:.15 and .55); 
\draw[red, thick, dashed] (5,-.4) arc (90:-90:.15 and .55)
node[pos=0.5,right]{\small{${}\!\alpha_{2}'$}}; 

\draw[red, thick] (9.5,1.5) arc (90:270:.15 and .55); 
\draw[red, thick,->] (9.5,.4) arc (270:175:.15 and .55); 
\draw[red, thick, dashed] (9.5,1.5) arc (90:-90:.15 and .55)
node[pos=0.5,right]{\small{${}\!\alpha_{n}$}}; 

\draw[red, thick] (9.5,-.4) arc (90:270:.15 and .55); 
\draw[red, thick,->] (9.5,-.4) arc (90:185:.15 and .55); 
\draw[red, thick, dashed] (9.5,-.4) arc (90:-90:.15 and .55)
node[pos=0.5,right]{\small{${}\!\alpha_{n}'$}}; 

\draw[red, thick] (-1.5,0) arc (180:360:.55 and .15)
node[pos=0.5,below]{\small{${}\!\alpha_0$}}; 
\draw[red, thick,->] (-.4,0) arc (0:-95:.55 and .15); 
\draw[red, thick, dashed] (-.4,0) arc (0:180:.55 and .15);

\draw[black, ultra thick] (0,1.5) -- (15,1.5) arc (90:-90:1.5) -- (0,-1.5) arc (270:90:1.5);
\draw[black, ultra thick] (0,0) circle (.4);
\draw[black, ultra thick] (2.5,0) circle (.4);
\draw[black, ultra thick] (5,0) circle (.4);
\draw[black, ultra thick] (9.5,0) circle (.4);
\draw[black, ultra thick] (12,0) circle (.4);
\draw[black, ultra thick] (15,0) circle (.4);

\fill[black] (6.95,0) circle (1.5pt);
\fill[black] (7.25,0) circle (1.5pt);
\fill[black] (7.55,0) circle (1.5pt);
\fill[black] (13.2,0) circle (1.5pt);
\fill[black] (13.5,0) circle (1.5pt);
\fill[black] (13.8,0) circle (1.5pt);

\end{tikzpicture}
\caption{Type two $\alpha$-multicurve}\label{fig_alpha2}
\end{figure}

\begin{definition}\label{defin_alpha2}
A \textit{type two $\alpha$-multicurve} is a $(2n+1)$-component oriented multicurve
\begin{equation*}
N=\alpha_0\cup\alpha_1\cup\alpha_1'\cup\cdots\cup\alpha_{n}\cup\alpha_{n}',
\end{equation*}
where $1\le n\le g-3$, that satisfies the following conditions (see Fig.~\ref{fig_alpha2}):
\begin{enumerate}
 \item all components of~$N$ are non-separating simple closed curves,
 \item $\alpha_i'$ is homologous to~$\alpha_i$ for $i=1,\ldots,n$,
 \item $S_g\setminus N$ consists of~$n+1$ connected components~$Y_1,\ldots,Y_{n+1}$  such that
 \begin{itemize}
\item for every $i=1,\ldots,n$, the component~$Y_i$ is a four-punctured sphere that is adjacent to~$\alpha_i$ and~$\alpha_{i-1}'$ from the left and is adjacent to~$\alpha_i'$ and~$\alpha_{i-1}$ from the right, where we use the convention $\alpha_0'=\alpha_0$,
\item $Y_{n+1}$ is a twice-punctured surface of genus $g'=g-n-1$ that is adjacent to~$\alpha_n$ from the right and to~$\alpha_n'$ from the left,
 \end{itemize}
 \item $x=l_0[\alpha_0]+\cdots+l_{n}[\alpha_{n}]$ for some positive integers $l_0,\ldots,l_{n}$.
\end{enumerate}
\end{definition}

From properties~(4) in Definitions~\ref{defin_alpha} and~\ref{defin_alpha2} it follows that any $\alpha$-multicurve~$N$ belongs to~$\M$, that is, defines a cell~$P_N$ of the complex of cycles~$\B_g$. Moreover, the cell~$P_N$ is isomorphic to the cube~$[0,1]^{g-2}$ for a type one $\alpha$-multicurve and to the cube~$[0,1]^n$ for a $(2n+1)$-component type two $\alpha$-multicurve. The isomorphisms $\varphi\colon [0,1]^{g-2}\to P_N$ (for type one) and $\varphi\colon [0,1]^n\to P_N$ (for type two) are given by
\begin{equation}\label{eq_isomorphism}
\varphi(t_1,\ldots,t_{g-2})= l_0\alpha_0+l_{g-1}\alpha_{g-1}+\sum_{i=1}^{g-2}l_i\bigl((1-t_i)\alpha_i+t_i\alpha_i'\bigr)
\end{equation}
and
\begin{equation}\label{eq_isomorphism2}
\varphi(t_1,\ldots,t_n)= l_0\alpha_0+\sum_{i=1}^{n}l_i\bigl((1-t_i)\alpha_i+t_i\alpha_i'\bigr),
\end{equation}
respectively.

It follows immediately from the definitions that, for any $\alpha$-multicurve~$N$ and any $h\in \I_g$, the multicurve~$h(N)$ is also an $\alpha$-multicurve. So the Torelli group~$\I_g$ acts on the set of isotopy classes of $\alpha$-multicurves.
To apply the construction from the previous section to $\alpha$-multicurves we need first to prove the following proposition.

\begin{propos}\label{propos_alpha_perfect}
For any $\alpha$-multicurve~$N$, the $\I_g$-orbit $\fN=\I_gN$ is perfect.
\end{propos}

\begin{proof}
Assume that $M\in\M$ is an oriented multicurve containing two different submulticurves that belong to the orbit~$\fN$. Acting by certain element of~$\I_g$, we may achieve that one of those two submulticurves is $N$ itself. Let~$\widetilde{N}\ne N$ be another submulticurve of~$M$ belonging to~$\fN$; then $\widetilde{N}=h(N)$ for a mapping class $h\in\I_g$. Denote the components of~$N$ as in Definitions~\ref{defin_alpha} or~\ref{defin_alpha2}. Since the multicurves~$N$ and~$\widetilde{N}$ consist of the same number of components, we see that~$\widetilde{N}$ contains a component~$\gamma$ that is not a component of~$N$. Then $\gamma$ is contained in one of the components~$Y_i$ of $S_g\setminus N$. If $Y_i$ were a four-punctured sphere, then $\gamma$ would cut $Y_i$ into two three-punctured spheres and hence would not be homologous to any of the curves~$\alpha_j$. So we would obtain a contradiction, since any component of~$\widetilde{N}$ must be homologous to one of the components of~$N$. If $N$ is a type one $\alpha$-multicurve, the proposition already follows, since all components of $S_g\setminus N$ are four-punctured spheres.

Now, suppose that $N$ is a $(2n+1)$-component type two $\alpha$-multicurve. Then it follows from the above reasoning that any component~$\gamma$ of~$\widetilde{N}$ that is not a component of~$N$ lies in $Y_{n+1}$ and is homologous to~$\alpha_n$. Therefore, $h(\alpha_i)=\alpha_i$ for $i=0,\ldots,n-1$, $h(\alpha'_i)=\alpha'_i$ for $i=1,\ldots,n-1$, and each of the curves~$h(\alpha_n)$ and~$h(\alpha_n')$ either coincides with one of the curves~$\alpha_n$ and~$\alpha_n'$ or lies in~$Y_{n+1}$. Since $h(Y_{n+1})$ is a connected component of~$S_g\setminus \bigl(h(\alpha_n)\cup h(\alpha_n')\bigr)$ and $h(Y_{n+1})$ does not contain~$h(\alpha_0)=\alpha_0$, it follows easily that $h(Y_{n+1})\subseteq Y_{n+1}$. Swapping $N$ and~$\widetilde{N}$ in the above reasoning, we likewise get that $Y_{n+1}\subseteq h(Y_{n+1})$. Therefore $h(Y_{n+1})=Y_{n+1}$. Consequently, $h(\alpha_n)=\alpha_n$ and~$h(\alpha_n')=\alpha_n'$ and hence $\widetilde{N}=N$, which contradicts the assumption made.
\end{proof}

\section{The spectral sequences~$\hE^*_{*,*}(\fN)$ corresponding to $\alpha$-mulicurves}
\label{section_alpha_E}

From Proposition~\ref{propos_alpha_perfect} it follows that, for each $\I_g$-orbit~$\fN=\I_gN$ of $\alpha$-multicurves, the construction from Section~\ref{section_aux_spectral} yields a well-defined spectral sequence~$\hE^*_{*,*}(\fN)$ with
$$
\hE^1_{p,q}(\fN)\cong\bigoplus_{M\in\M_p\colon M\supseteq N}\HH_q(\Stab_{H_N}(M))
$$
and a morphism of spectral sequences
$$
\hPi_{\fN}\colon E^*_{*,*}\to\hE^*_{*,*}(\fN).
$$
Our next goal is to study this spectral sequence. To do this, we need to study the homology of the groups~$\Stab_{H_N}(M)$ for multicurves~$M\in\M$ that contain~$N$. More precisely, we will study  the homology of the group~$H_N=\Stab_{H_N}(N)$ in detail, while for multicurves~$M$ stirctly containing~$N$ we will only estimate the cohomological dimension of~$\Stab_{H_N}(M)$. We consecutively consider the cases of $\alpha$-multicurves of types one and two.

\subsection{Type one}
Let $N$ be a type one $\alpha$-multicurve with components denoted as in Definition~\ref{defin_alpha}. Recall that, by definition, the group~$H_N$ is the image of the homomorphism
$$
j_N\colon \Stab_{\I_g}(N)\to\PMod(S_g\setminus N).
$$
We have
$$
\PMod (S_g\setminus N)=\PMod(Y_1)\times\cdots\times\PMod(Y_{g-1}).
$$

 Set $a_i=[\alpha_i]$ for $i=0,\ldots,g-1$.
 For every $i=1,\ldots,g-1$, take a simple closed curve~$\delta_i\subset Y_i$ separating $\alpha_{i-1}$ and~$\alpha'_{i-1}$ from $\alpha_i$ and~$\alpha_i'$ and a simple closed curve $\gamma_i\subset Y_i$ that represents the homology class~$a_i-a_{i-1}$ and has geometric intersection number~$2$ with~$\delta_i$, see Fig.~\ref{fig_gamma_delta}.

\begin{figure}
\begin{tikzpicture}[scale=.8]

\draw[mygreen, thick] (1.5,-1.5) arc (270:90:0.3 and 1.5);
\draw[mygreen, thick, dashed] (1.5,-1.5) arc (-90:90:0.3 and 1.5)
node[pos=0.21,right]{${}\!\delta_i$};

\draw[violet, thick] (.4,0) arc (180:360:1.1 and .25)
node[pos=0.8,below]{$\gamma_i$};
\draw[violet, thick,->] (2.6,0) arc (0:-90:1.1 and .25);
\draw[violet, thick, dashed] (.4,0) arc (180:0:1.1 and .25);

\draw[red, very thick] (3,1.5) arc (90:270:.15 and .55); 
\draw[red, very thick,->] (3,.4) arc (270:175:.15 and .55); 
\draw[red, very thick] (3,1.5) arc (90:-90:.15 and .55)
node[pos=0.5,right]{\small{${}\!\alpha_{i}$}}; 

\draw[red, very thick] (3,-.4) arc (90:270:.15 and .55); 
\draw[red, very thick,->] (3,-.4) arc (90:185:.15 and .55); 
\draw[red, very thick] (3,-.4) arc (90:-90:.15 and .55)
node[pos=0.5,right]{\small{${}\!\alpha_{i}'$}}; 

\draw[red, very thick] (0,1.5) arc (90:270:.15 and .55); 
\draw[red, very thick,->] (0,.4) arc (270:175:.15 and .55); 
\draw[red, very thick, dashed] (0,1.5) arc (90:-90:.15 and .55)
node[pos=0.5,right]{\small{${}\!\alpha_{i-1}$}}; 

\draw[red, very thick] (0,-.4) arc (90:270:.15 and .55); 
\draw[red, very thick,->] (0,-.4) arc (90:185:.15 and .55); 
\draw[red, very thick, dashed] (0,-.4) arc (90:-90:.15 and .55)
node[pos=0.5,right]{\small{${}\!\alpha_{i-1}'$}};

\draw[black, ultra thick] (0,1.5) -- (3,1.5);
\draw[black, ultra thick] (0,-1.5) -- (3,-1.5);
\draw[black, ultra thick] (0,-.4) arc (-90:90:.4);
\draw[black, ultra thick] (3,.4) arc (90:270:.4);

\end{tikzpicture}
\caption{Curves~$\delta_i$ and~$\gamma_i$}\label{fig_gamma_delta}
\end{figure}

Denote by~$\F_{\infty}$ a free group with countably many generators.

\begin{propos}\label{propos_stabilizer}
We have
\begin{equation}\label{eq_HN}
H_N= \underbrace{\F_{\infty}\times\cdots\times\F_{\infty}}_{g-1 \text{ \textnormal{factors}}},
\end{equation}
where, for each~$i$,  the $i$\textsuperscript{th} factor~$\F_{\infty}$ is the subgroup of~$\PMod(Y_i)$ generated freely by the elements $y_{i,k}=T_{\gamma_i}^kT_{\delta_i}T_{\gamma_i}^{-k}$ with $k\in\Z$.\end{propos}

\begin{proof}
Every~$\PMod(Y_i)$ is the free group with two generators~$T_{\delta_i}$ and~$T_{\gamma_i}$. Consider the homomorphism $\kappa_i\colon\PMod(Y_i)\to\Z$ such that $\kappa_i(T_{\delta_i})=0$ and~$\kappa_i(T_{\gamma_i})=1$. It is a standard fact that the kernel of this homomorphism is the group~$\F_{\infty}$ generated freely by the elements~$y_{i,k}=T_{\gamma_i}^kT_{\delta_i}T_{\gamma_i}^{-k}$, where $k$ runs over~$\Z$. Let
$$
\kappa\colon \PMod(S_g\setminus N)\to \Z^{g-1}
$$
be the direct product of the homomorphisms~$\kappa_1,\ldots,\kappa_{g-1}$. Then the kernel of~$\kappa$ is exactly the direct product $\F_{\infty}\times\cdots\times\F_{\infty}$ in the right-hand side of~\eqref{eq_HN}. So we need to prove that  $H_N=\ker\kappa$. Since $y_{i,k}$ is the Dehn twist about the separating simple closed curve~$T_{\gamma_i}^k(\delta_i)$, which is disjoint from~$N$, we see that $y_{i,k}\in H_N$. Hence $\ker\kappa\subseteq H_N$.

Let us prove that this inclusion is in fact an equality. The mapping classes
$$
T_{\bgamma}^{\bk}=T_{\gamma_i}^{k_1}\cdots T_{\gamma_{g-1}}^{k_{g-1}},
$$
where $\bk=(k_1,\ldots,k_{g-1})$ runs over~$\Z^{g-1}$, are representatives of all cosets of the normal subgroup $\ker\kappa$ in~$\PMod(S_g\setminus N)$. So to prove that $H_N=\ker\kappa$ it is sufficient to show that $T_{\bgamma}^{\bk}\notin H_N$ unless $\bk=0$.

Recall that there is the following Birman--Lubotzky-McCarthy short exact sequence (see~\cite[Lemma~2.1]{BLM83}):
$$
1\to G(N)\to \Stab_{\Mod(S_g)}\bigl(\vec{N}\bigr)\to \PMod(S_g\setminus N)\to 1.
$$
Here $G(N)$ is the abelian group generated by the Dehn twists about components of~$N$ and~$\Stab_{\Mod(S_g)}\bigl(\vec{N}\bigr)$ is the subgroup of~$\Mod(S_g)$ consisting of all mapping classes that stabilize every component of~$N$ and preserve the orientation of every component of~$N$. The preimage of~$T_{\bgamma}^{\bk}$ in the group~$\Stab_{\Mod(S_g)}\bigl(\vec{N}\bigr)$ consists of all mapping classes
\begin{equation}\label{eq_preimage_of_fk}
T_{\gamma_i}^{k_1}\cdots T_{\gamma_{g-1}}^{k_{g-1}}h,
\end{equation}
where $h$ runs over~$G(N)$. From the definition of~$H_N$ it follows that $T_{\bgamma}^{\bk}\in H_N$ if and only if at least one of the mapping classes~\eqref{eq_preimage_of_fk} lies in the Torelli group~$\I_g$.

Vautaw~\cite[Theorem~3.1]{Vau02} proved that, for any multicurve~$K$, a product of powers of Dehn twists about components of~$K$ lies in the Torelli group if and only if it is a product of BP maps corresponding to bounding pairs of components of~$K$. Applying this result to the multicurve $K=N\cup\gamma_1\cup\cdots\cup\gamma_{g-1}$ and taking into account that none of the curves~$\gamma_i$ is homologous to any other of them or to a component of~$N$, we obtain that a mapping class of the form~\eqref{eq_preimage_of_fk} may belong to the Torelli group only if $\bk=0$. This completes the proof of the proposition.
\end{proof}

The next corollary follows immediately from Proposition~\ref{propos_stabilizer} and the K\"unneth theorem.

\begin{cor}\label{cor_Hg-1}
The cohomological dimension of the group~$H_N$ is equal to~$g-1$.
The top homology group~$\HH_{g-1}(H_N)$ is the free abelian group with basis consisting of the abelian cycles~$\CA(y_{1,k_1},\ldots,y_{g-1,k_{g-1}})$, where $\bk=(k_1,\ldots,k_{g-1})$ runs over~$\Z^{g-1}$.
\end{cor}

The obtained description of the top homology group~$\HH_{g-1}(H_N)$ depends on the choice of curves~$\delta_i$ and~$\gamma_i$. We would like to describe this homology group in more invariant terms.  Any choice of separating curves~$\delta_i$ in~$Y_i$ provides an orthogonal splitting
\begin{equation}\label{eq_H1split}
\HH_1(S_g)=U_0\oplus\cdots\oplus U_{g-1}
\end{equation}
such that $a_i\in U_i$ for $i=0,\ldots,g-1$. This splitting belongs to~$\CS$ and
$$
\Delta=\delta_1\cup\cdots\cup\delta_{g-1}
$$
is a $\delta$-multicurve for it, see Definition~\ref{defin_delta}.
We denote by~$\CS(N)$ the set of all splittings~\eqref{eq_H1split} corresponding to various choices of a $\delta$-multicurve~$\Delta$ with $\delta_i\subset Y_i$.

\begin{propos}\label{propos_indep}
Let $\Delta=\delta_1\cup\cdots\cup\delta_{g-1}$ and~$\widetilde{\Delta}=\tilde\delta_1\cup\cdots\cup\tilde\delta_{g-1}$ be two $\delta$-multicurves with $\delta_i\subset Y_i$ and $\tilde{\delta}_i\subset Y_i$ for all~$i$. Suppose that $\Delta$ and~$\widetilde{\Delta}$ give the same splitting~\eqref{eq_H1split} in homology. Then
 $$
 \CA\bigl(T_{\delta_1},\ldots,T_{\delta_{g-1}}\bigr)=\CA\bigl(T_{\tilde\delta_1},\ldots,T_{\tilde\delta_{g-1}}\bigr)
 $$
 in $\HH_{g-1}(H_N)$.
\end{propos}

\begin{proof}
  Extend $a_0,\ldots,a_{g-1}$ to a symplectic basis
  $$
  a_0,b_0,\ldots,a_{g-1},b_{g-1}
  $$
  of $\HH_1(S_g)$ so that $a_i, b_i$ is a symplectic basis of~$U_i$ for each~$i$.
  Using the Alexander method (cf.~\cite[Section~2.3]{FaMa12}), one easily obtains that there exists a mapping class~$f\in\Mod(S_g)$ that stabilizes every component of~$N$ (preserving the orientation of it) and takes every~$\delta_i$ to~$\tilde\delta_i$. Then $f$ stabilizes every homology class~$a_i$ and every summand~$U_i$. Hence, $f_*(b_i)=b_i+m_ia_i$ for some $m_i\in\Z$. Then the mapping class $$h=fT_{\alpha_0}^{m_0}\cdots T_{\alpha_{g-1}}^{m_{g-1}}$$
  lies in $\Stab_{\I_g}(N)$ and takes every~$\delta_i$ to~$\tilde\delta_i$.
Set $q=j_N(h)$. Then $q\in H_N$ and $T_{\tilde\delta_i}=qT_{\delta_i}q^{-1}$ for all~$i$. The proposition follows, since any group acts trivially on the homology of itself.
\end{proof}

For a splitting $\CU\in\CS(N)$, we put
$$
A_{\CU}=\CA(T_{\delta_1},\ldots,T_{\delta_{g-1}})\in \HH_{g-1}(H_N),
$$
where $\Delta=\delta_1\cup\cdots\cup\delta_{g-1}$ is a $\delta$-multicurve with $\delta_i\subset Y_i$ giving the splitting~$\CU$ in homology.
By Proposition~\ref{propos_indep} the homology class~$A_{\CU}$ is independent of the choice of~$\Delta$.

Now, suppose that curves~$\delta_i$ and~$\gamma_i$ are as in Proposition~\ref{propos_stabilizer} (see Fig.~\ref{fig_gamma_delta}). As in the proof of those proposition, we introduce notation
$$
T_{\bgamma}^{\bk}=T_{\gamma_1}^{k_1}\cdots T_{\gamma_{g-1}}^{k_{g-1}},
$$
where $\bk=(k_1,\ldots,k_{g-1})\in\Z^{g-1}$. Then
$$
T_{\bgamma}^{\bk}(\Delta)=T_{\gamma_1}^{k_1}(\delta_1)\cup\cdots\cup T_{\gamma_{g-1}}^{k_{g-1}}(\delta_{g-1})
$$
are $\delta$-multicurves disjoint from~$N$.

\begin{propos}\label{propos_diff_split}
 The multicurves~$T_{\bgamma}^{\bk}(\Delta)$ with $\bk\in\Z^{g-1}$ give pairwise different splittings~$\CU^{(\bk)}$ in homology. Moreover, the splittings~$\CU^{(\bk)}$ exhaust the whole set~$\CS(N)$.
\end{propos}

\begin{proof}
The fact that the splittings~$\CU^{(\bk)}$ are pairwise different can be checked by a direct computation. However, it also can be deduced from Corollary~\ref{cor_Hg-1} without any computation. Indeed, the mapping classes~$y_{i,k}=T_{\gamma_i}^kT_{\delta_i}T_{\gamma_i}^{-k}$ in Proposition~\ref{propos_stabilizer} are the Dehn twists about simple closed curves~$T_{\gamma_i}^k(\delta_i)$, so
$$
A_{\CU^{(\bk)}}=\CA(y_{1,k_1},\ldots,y_{g-1,k_{g-1}}).
$$
By Corollary~\ref{cor_Hg-1} these homology classes are linearly independent and hence pairwise different. Therefore, by Proposition~\ref{propos_indep} the splittings~$\CU^{(\bk)}$ are also pairwise different.

Let us now prove that the splittings~$\CU^{(\bk)}$ exhaust the whole set~$\CS(N)$. Suppose that $\widetilde{\CU}$ is an arbitrary splitting in~$\CS(N)$. Then $\widetilde{\CU}$ is given by a $\delta$-multicurve $\widetilde{\Delta}=\tilde\delta_1\cup\cdots\cup\tilde\delta_{g-1}$ with $\tilde\delta_i\subset Y_i$ for all~$i$. Since $Y_i$ is a four-punctured sphere and both curves~$\delta_i$ and~$\tilde\delta_i$ separate the punctures corresponding to~$\alpha_{i-1}$ and~$\alpha_{i-1}'$ from the punctures corresponding to~$\alpha_{i}$ and~$\alpha_{i}'$, we obtain that there exists a mapping class $q_i\in\PMod(Y_i)$ that takes~$\delta_i$ to~$\tilde\delta_i$. The product of the mapping classes~$q_i$ is a mapping class $q\in \PMod(S_g\setminus N)$ such that $q(\Delta)=\widetilde{\Delta}$. As in the proof of Proposition~\ref{propos_stabilizer}, we see that the elements~$T_{\bgamma}^{\bk}$ form a system of representatives of all cosets of~$H_N$ in~$\PMod(S_g\setminus N)$. Hence, one of the elements~$T_{\bgamma}^{\bk}$ lies in the same coset as~$q$. Then the mapping class $qT_{\bgamma}^{-\bk}$ belongs to~$H_N$ and takes the multicurve~$T_{\bgamma}^{\bk}(\Delta)$ to~$\widetilde{\Delta}$. By the definition of the group~$H_N$, there exists a mapping class $h\in\Stab_{\I_g}(N)$ satisfying $j_N(h)= qT_{\bgamma}^{-\bk}$. The mapping class~$h$ belongs to the Torelli group and takes~$T_{\bgamma}^{\bk}(\Delta)$ to~$\widetilde{\Delta}$. Therefore the multicurves~$T_{\bgamma}^{\bk}(\Delta)$ and~$\widetilde{\Delta}$ give the same splitting in homology. Thus, $\widetilde{\CU}=\CU^{(\bk)}$, which completes the proof of the proposition.
\end{proof}

We now can reformulate the second assertion of Corollary~\ref{cor_Hg-1} in the following invariant form.

\begin{cor} \label{cor_Hg-1inv}
The top homology group~$\HH_{g-1}(H_N)$ is the free abelian group with basis consisting of homology classes~$A_{\CU}$, where $\CU$ runs over the set~$\CS(N)$.
\end{cor}

Recall that the cell~$P_N$ of the complex of cycles is isomorphic to the cube~$[0,1]^{g-2}$, so $N\in\M_{g-2}$, see~\eqref{eq_isomorphism}. Let us now study the groups~$\Stab_{H_N}(M)$ for oriented multicurves~$M$ that contain~$N$ and belong to~$\M$.

\begin{propos}\label{propos_HM}
Suppose that $M\in\M_m$ and~$M\supseteq N$. Then
$$
\Stab_{H_N}(M)\cong \underbrace{\F_{\infty}\times\cdots\times\F_{\infty}}_{2g-3-m\text{ \textnormal{factors}}}
$$
In particular, $\cd(\Stab_{H_N}(M))= 2g-3-m.$
\end{propos}

\begin{proof}
Recall that we denote by~$D(M)$ the rank of the subgroup of~$\HH_1(S_g)$ generated by the homology classes of the components of~$M$. Since this subgroup is isotropic with respect to the intersection form, we have $D(M)\le g$. On the other hand, $D(M)\ge D(N)=g$. Hence $D(M)=g$. Then formula~\eqref{eq_dim} gives
$$|M|=m+g=|N|+m-g+2.$$
Let $\varepsilon_1,\ldots,\varepsilon_{m-g+2}$ be the components of~$M$ that are not components of~$N$.  Since the components of~$M$ are pairwise non-isotopic, we see that two different components~$\varepsilon_s$ and~$\varepsilon_t$ cannot lie in the same component~$Y_i$ of~$S_g\setminus N$.
Suppose that $\varepsilon_s$ lies in~$Y_{i_s}$ for $s=1,\ldots,m-g+2$, and put
$$
I=\{i_1,\ldots,i_{m-g+2}\},\qquad I'=\{1,\ldots,g-1\}\setminus I.
$$
The curve~$\varepsilon_s$ decomposes~$Y_{i_s}$ into two pieces either of which is a three-punctured sphere. Hence $\PMod(Y_{i_s}\setminus\varepsilon_s)=1$. Therefore no non-trivial element of~$\PMod(Y_{i_s})$ fixes a curve~$\varepsilon_s$.
It follows easily that the subgroup~$\Stab_{H_N}(M)\subseteq H_N$ is exactly the direct product of the $2g-3-m$ factors with numbers in~$I'$ in the decomposition~\eqref{eq_HN}.
\end{proof}

Now, let $\fN$ denote the $\I_g$-orbit of the multicurve~$N$.
Since $N\in\M_{g-2}$, the following proposition is a direct consequence of formulae~\eqref{eq_hE1}--\eqref{eq_hE1_N} and Proposition~\ref{propos_HM}.
\begin{propos}
$\hE^1_{p,q}(\fN)=0$ whenever either  $p+q>2g-3$ or $p<g-2$.
\end{propos}

\begin{cor}\label{cor_hE_iso}
All differentials~$d^r$ with~$r\ge 1$ of the spectral sequence~$\hE^*_{*,*}(\fN)$ either from or to the groups $\hE^r_{g-2,\,g-1}(\fN)$ are trivial. Hence
$$
\hE^{\infty}_{g-2,\,g-1}(\fN)=\hE^1_{g-2,\,g-1}(\fN)\cong \HH_{g-1}(H_N).
$$
\end{cor}

\subsection{Type two}\label{subsection_type2}
Let $N$ be a  $(2n+1)$-component type two $\alpha$-multicurve with components denoted as in Definition~\ref{defin_alpha2}.
Set $a_i=[\alpha_i]$ for $i=0,\ldots,n$. As for a type one $\alpha$-multicurve,
 for every $i=1,\ldots,n$, we take a simple closed curve~$\delta_i\subset Y_i$ separating $\alpha_{i-1}$ and~$\alpha'_{i-1}$ from $\alpha_i$ and~$\alpha_i'$, and a simple closed curve $\gamma_i\subset Y_i$ that represents the homology class~$a_i-a_{i-1}$ and has geometric intersection number~$2$ with~$\delta_i$, see Fig.~\ref{fig_gamma_delta}.

Recall that $Y_{n+1}$ is the oriented surface of genus~$g'=g-n-1$ with two punctures. Since $n\le g-3$, we have $g'\ge 2$. Let $\Sigma$ be the one-point compactification of~$Y_{n+1}$. Then $\Sigma$ is homeomorphic to a closed oriented surface of genus~$g'$ with two points identified. Since every homeomorphism of~$Y_{n+1}$ onto itself extends uniquely to~$\Sigma$, we obtain that the group~$\PMod(Y_{n+1})$ acts naturally on $\HH_1(\Sigma)\cong\Z^{2g'+1}$. Let~$G_N$ be the kernel of this action. The following analog of Proposition~\ref{propos_stabilizer} describies the subgroup
$$
H_N\subseteq\PMod(S_g\setminus N)=\PMod(Y_1)\times\cdots\times\PMod(Y_n)\times\PMod(Y_{n+1}).
$$

\begin{propos}\label{propos_stabilizer2}
We have
\begin{equation}\label{eq_HN2}
H_N= {\underbrace{\F_{\infty}\times\cdots\times\F_{\infty}}_{n \text{ \textnormal{factors}}}}\times G_N,
\end{equation}
where, for each~$i$,  the $i$\textsuperscript{th} factor~$\F_{\infty}$ is the subgroup of~$\PMod(Y_i)$ generated freely by the elements $y_{i,k}=T_{\gamma_i}^kT_{\delta_i}T_{\gamma_i}^{-k}$ with $k\in\Z$.
\end{propos}

\begin{proof}
We set
$$
Q_N=\PMod(Y_1)\times\cdots\times\PMod(Y_n).
$$
Literally as in the proof of Proposition~\ref{propos_stabilizer} we
\begin{itemize}
 \item construct a homomorphism $\kappa\colon Q_N\to\Z^n$ that takes each~$T_{\gamma_i}$ to the generator of the $i$\textsuperscript{th} factor~$\Z$ and each~$T_{\delta_i}$ to~$0$,
 \item prove that
 $$
 H_N\cap Q_N=\ker\kappa=\underbrace{\F_{\infty}\times\cdots\times\F_{\infty}}_{n \text{ factors}},
 $$
 where the $i$\textsuperscript{th} factor~$\F_{\infty}$ is freely generated by all~$y_{i,k}$ with $k
 \in\Z$.
\end{itemize}

The proposition will follow as soon as we manage to prove that $G_N$ is contained in~$H_N$ and the image of the composite map
\begin{equation}\label{eq_composite_HN}
\begin{tikzcd}
 H_N\arrow[hook]{r} & Q_N\times\PMod(Y_{n+1}) \arrow[rr, "\text{projection}" above] & &  \PMod(Y_{n+1})
\end{tikzcd}
\end{equation}
is contained in~$G_N$. Let us prove these two assertions in turn.

1. Let us prove that $G_N$ is contained in~$H_N$. Let~$\overline{Y}_{\!\!n+1}=Y_{n+1}\cup\alpha_n\cup\alpha_n'$ be the closure of~$Y_{n+1}$ in~$S_g$. Any mapping class $f\in \PMod(Y_{n+1})$ can be represented by an orientation preserving homeomorphism $\varphi\colon Y_{n+1}\to Y_{n+1}$ that extends to a homeomorphism $\overline{\varphi}\colon \overline{Y}_{\!\!n+1}\to \overline{Y}_{\!\!n+1}$ fixing every point of~$\alpha_n\cup\alpha_n'$. Extend~$\overline{\varphi}$ by the identity to a  homeomorphism $S_g\to S_g$ and let $h\in\Mod(S_g)$ be the mapping class represented by this homeomorphism. By the construction, $h$ stabilizes the homology class $a_n=[\alpha_n]$. Suppose that $f\in G_N$. Then  $h$ acts trivially on the quotient~$\HH_1(S_g)/\langle a_n\rangle$. It follows that the action of~$h$ on~$\HH_1(S_g)$ has the form $h_*(c)=c+s(c\cdot a_n) a_n$ for certain~$s\in\Z$ independent of~$c$. Then the mapping class~$hT_{\alpha_n}^{-s}$ belongs to~$\Stab_{\I_g}(N)$ and~$j_N(hT_{\alpha_n}^{-s})=f$. Therefore $f\in H_N$.

2. Let us prove that the image of the composite map~\eqref{eq_composite_HN} is contained in~$G_N$. Suppose that an element $qr$ belongs to~$H_N$, where $q\in Q_N$ and $r\in\PMod(Y_{n+1})$. We need to prove that $r$ lies in~$G_N$. We have $qr=j_N(h)$ for certain $h\in\Stab_{\I_g}(N)$.
Considering the commutative diagram
$$
\begin{tikzcd}
 Y_{n+1} \arrow[hook]{r} \arrow[hook]{d} & S_g  \arrow{d}  \\
 \Sigma & S_g/(S_g\setminus Y_{n+1}) \arrow[l,"\approx" above]
\end{tikzcd}
$$
we see that the action of~$r$ on~$\Sigma$ is the quotient of the action of~$h$ on~$S_g$. Since the quotient map $S_g\to\Sigma$ induces a surjective homomorphism $\HH_1(S_g)\to \HH_1(\Sigma)$ and $h$ is in the Torelli group, we obtain that $r$ acts trivially on~$\HH_1(\Sigma)$, that is, $r\in G_N$.
\end{proof}

We conveniently put
$$
F_N=H_N\cap Q_N=\underbrace{\F_{\infty}\times\cdots\times\F_{\infty}}_{n \text{ factors}}.
$$
Then $$H_N=F_N\times G_N.$$ We would like to compute the cohomological dimension of~$H_N$ and study the top homology group of it. To do this we need to compute the cohomological dimensions and study the top homology groups of both~$F_N$ and~$G_N$.

The study of the group~$F_N$ is completely similar to the study of the group~$H_N$ for a type one $\alpha$-multicurve~$N$ in the previous subsection. Let us formulate the analogs of Corollaries~\ref{cor_Hg-1} and~\ref{cor_Hg-1inv} and Propositions~\ref{propos_indep} and~\ref{propos_diff_split} in our present situation. We omit the proofs of these assertions, since they repeat literally the proofs of those Corollaries and Propositions.

\begin{cor}\label{cor_Hg-12}
The cohomological dimension of the group~$F_N$ is equal to~$n$.
The top homology group~$\HH_{n}(F_N)$ is the free abelian group with basis consisting of the abelian cycles~$\CA(y_{1,k_1},\ldots,y_{n,k_n})$, where $\bk=(k_1,\ldots,k_n)$ runs over~$\Z^n$.
\end{cor}

Any choice of separating curves~$\delta_i$ in~$Y_i$ for $i=1,\ldots,n$ gives a splitting
\begin{equation}\label{eq_H1split2}
\HH_1(S_g)=U_0\oplus\cdots\oplus U_n
\end{equation}
with $\rank U_i=2$ for $i=0,\ldots,n-1$ and $\rank U_n=2g-2n$. This splitting belongs to the set~$\CS_n$ defined in Subsection~\ref{subsection_A2}. We denote by $\CS_n(N)$ the set of all splittings that arise as splittings~\eqref{eq_H1split2} for various choices of the curves~$\delta_i$.

\begin{propos}\label{propos_indep2}
Let $\Delta=\delta_1\cup\cdots\cup\delta_n$ and~$\widetilde{\Delta}=\tilde\delta_1\cup\cdots\cup\tilde\delta_n$ be two $\delta$-multicurves with $\delta_i\subset Y_i$ and $\tilde{\delta}_i\subset Y_i$ for all~$i$. Suppose that $\Delta$ and~$\widetilde{\Delta}$ give the same splitting~\eqref{eq_H1split2} in homology. Then
 $$
 \CA\bigl(T_{\delta_1},\ldots,T_{\delta_n}\bigr)=\CA\bigl(T_{\tilde\delta_1},\ldots,T_{\tilde\delta_n}\bigr)
 $$
 in $\HH_n(F_N)$.
\end{propos}

Now, for each splitting $\CU\in\CS_n(N)$, we have a well-defined homology class
$$
A_{\CU}=\CA(T_{\delta_1},\ldots,T_{\delta_n})\in \HH_n(F_N),
$$
where $\Delta=\delta_1\cup\cdots\cup\delta_n$ is a $\delta$-multicurve with $\delta_i\subset Y_i$ giving the splitting~$\CU$ in homology.

Suppose that curves~$\delta_i$ and~$\gamma_i$ are as in Proposition~\ref{propos_stabilizer2} (see Fig.~\ref{fig_gamma_delta}). We put
$$
T_{\bgamma}^{\bk}=T_{\gamma_1}^{k_1}\cdots T_{\gamma_n}^{k_n},
$$
where $\bk=(k_1,\ldots,k_n)\in\Z^n$. Then
$$
T_{\bgamma}^{\bk}(\Delta)=T_{\gamma_1}^{k_1}(\delta_1)\cup\cdots\cup T_{\gamma_n}^{k_n}(\delta_n)
$$
is again a $\delta$-multicurves with the components in $Y_1,\ldots,Y_n$, respectively.

\begin{propos}\label{propos_diff_split2}
 The multicurves~$T_{\bgamma}^{\bk}(\Delta)$ with $\bk\in\Z^n$ give pairwise different splittings~$\CU^{(\bk)}$ in homology. Moreover, the splittings~$\CU^{(\bk)}$ exhaust the whole set~$\CS_n(N)$.
\end{propos}

\begin{cor} \label{cor_Hg-1inv2}
The top homology group~$\HH_n(F_N)$ is the free abelian group with basis consisting of homology classes~$A_{\CU}$, where $\CU$ runs over the set~$\CS_n(N)$.
\end{cor}

Let us now study the group~$G_N$. Note that, though we have included the restriction $n\ge 1$ in the definition of a type two $\alpha$-multicurve, we have never used this restriction in the proof of Proposition~\ref{propos_stabilizer2}. So this proposition remains true for $n=0$ as well, that is, for a multicurve~$N$ consisting of a single component.

Consider an auxiliary oriented closed genus~$g-n$ surface~$S_{g-n}$, a non-separating simple closed curve~$v$ on it, and choose an orientation-preserving homeomorphism $Y_{n+1}\approx S_{g-n}\setminus v$. From the exact sequence~\eqref{eq_BLMC} for a multicurve consisting of a single component~$v$ it follows that the homomorphism
$$
j_v\colon \Stab_{\I_{g-n}}(v)\to \PMod(S_{g-n}\setminus v)=\PMod(Y_{n+1})
$$
is injective. Then Proposition~\ref{propos_stabilizer2} for this one-component multicurve reads as follows.

\begin{propos}\label{propos_iso_G}
 The image of the homomorphism~$j_v$ coincides with the subgroup $G_N\subset \PMod(Y_{n+1})$, so
 $$
 j_v\colon \Stab_{\I_{g-n}}(v)\to G_N
 $$
 is an isomorphism.
\end{propos}

\begin{remark}
The facts that the subgroup~$G_N\subseteq \PMod(Y_{n+1})$ consisting of all mapping classes that act trivially on~$\HH_1(\Sigma)$ is a factor in~\eqref{eq_HN2} and coincides with the image of~$j_v$  are special cases of a result by Putman~\cite[Theorem~3.3]{Put07}. The only  difference is that Putman considered surfaces with boundary components rather than with punctures.
\end{remark}

Now, Theorem~\ref{theorem_BBM2} reads as follows.

\begin{cor}\label{cor_G_homology}
 The cohomological dimension of the group~$G_N$ is equal to~$3g-5-3n$. The top homology group~$\HH_{3g-5-3n}(G_N)$ contains a free abelian subgroup of infinite rank.
\end{cor}

Combining Corollaries~\ref{cor_Hg-12}, \ref{cor_Hg-1inv2}, and~\ref{cor_G_homology},  we finally get the following result.

\begin{propos}\label{propos_H_homology_total}
The cohomological dimension of the group~$H_N$ is equal to~$3g-5-2n$. The top homology group~$\HH_{3g-5-2n}(H_N)$ contains a free abelian subgroup of infinite rank. Moreover,
$$
\HH_{3g-5-2n}(H_N)\cong \HH_n(F_N)\otimes \HH_{3g-5-3n}(G_N).
$$
So if $\{\zeta_{\lambda}\}$ is a system of linearly independent Bestvina--Bux--Margalit classes in
$$
\HH_{3g-5-3n}(G_N)\cong \HH_{3g-5-3n}\bigl(\Stab_{\I_{g-n}}(v)\bigr)
$$
indexed by elements~$\lambda$ of certain infinite set~$\Lambda$, then the homology classes $A_{\CU}\times \zeta_{\lambda}$ are linearly independent in~$\HH_{3g-5-2n}(H_N)$, where $(\CU,\lambda)$ runs over $\CS_n(N)\times\Lambda$.
\end{propos}

\begin{remark}
 Note that
 $$
 A_{\CU}\times\zeta_{\lambda}=\CA(T_{\delta_1},\ldots,T_{\delta_n};\zeta_{\lambda})
 $$
 whenever a $\delta$-multicurve $\delta_1\cup\cdots\cup\delta_n$ gives the splitting~$\CU$ in homology.
\end{remark}

We are now going to obtain an analog of Proposition~\ref{propos_HM}. We start with the following proposition, which is substantially contained in~\cite{BBM07}. For the convenience of the reader, we explain how to extract its proof from~\cite{BBM07}.

\begin{propos}\label{propos_est_HM}
Suppose that $M\in\M_m$. Then
 \begin{equation}\label{eq_est_HM}
 \cd(H_M)\le 3g-5-m-\bp(M),
 \end{equation}
 where $\bp(M)$ is the rank of the free abelian group~$\BP(M)$.
\end{propos}

\begin{proof}
Denote  the  connected components of~$S_g\setminus M$ by $R_1,\ldots,R_{P+Z}$ so that $R_i$ has genus $g_i\ge 1$ for $1\le i\le P$ and $R_i$ has genus $0$ for $P+1\le i\le P+Z$. Let $p_i$ be the number of punctures of~$R_i$. The proof of Lemma~6.12 in~\cite{BBM07} contains an estimate
\begin{equation*}
\cd(H_M)\le \sum_{i=1}^P(3g_i+p_i-4)+ \sum_{i=P+1}^{P+Z}(p_i-3).
\end{equation*}
Further, in the proof of Lemma~6.13 in~\cite{BBM07} it is shown that the right-hand side of this inequality is equal to
$3g-3-P-|M|$. By~\eqref{eq_dim} we have $|M|=m+D(M)$ and by Lemma~6.14 in~\cite{BBM07} we have $D(M)+P\ge \bp(M)+2$. Combining all these results, we immediately obtain the required estimate~\eqref{eq_est_HM}.
\end{proof}

\begin{cor}\label{cor_HM2}
Suppose that $M\in\M_m$ and~$M\supseteq N$. Then
 $$\cd(\Stab_{H_N}(M))\le 3g-5-m-n.$$
\end{cor}

\begin{proof}
The short exact sequences~\eqref{eq_BLMC} for the multicurves~$N$ and~$M$ give a commutative diagram with exact rows
\begin{equation*}
 \begin{tikzcd}
  1 \arrow[r] & \BP(N) \arrow[r] \arrow[d, hookrightarrow] &
  \Stab_{\I_g}(N) \arrow[r, "j_N"] & H_N \arrow[r] & 1 \\
  1 \arrow[r] & \BP(M) \arrow[r]  &
  \Stab_{\I_g}(M) \arrow[u, hookrightarrow] \arrow[r, "j_M"] & H_M \arrow[r] & 1
 \end{tikzcd}
\end{equation*}
in which vertical arrows are inclusions. The group~$\Stab_{H_N}(M)$ coincides with the image of~$\Stab_{\I_g}(M)$ under~$j_N$. Hence we obtain a short exact sequence
$$
1\to \BP(M)/\BP(N)\to\Stab_{H_N}(M)\to H_M\to 1.
$$
Since $\BP(N)$ is a direct summand of the free abelian group~$\BP(M)$, we see that the quotient group $\BP(M)/\BP(N)$ is a free abelian group of rank $\bp(M)-\bp(N)=\bp(M)-n$. Hence $\cd(\BP(M)/\BP(N))=\bp(M)-n$. Now, the required estimate follows from Proposition~\ref{propos_est_HM} and the  subadditivity of cohomological dimension (cf.~\cite[Proposition~VIII.2.4(b)]{Bro82}).
\end{proof}

Now, let $\fN$ denote the $\I_g$-orbit of the multicurve~$N$.
Since $N\in\M_n$, the following proposition is a direct consequence of formulae~\eqref{eq_hE1}--\eqref{eq_hE1_N} and Corollary~\ref{cor_HM2}.
\begin{propos}
$\hE^1_{p,q}(\fN)=0$ whenever either  $p+q>3g-5-n$ or $p<n$.
\end{propos}

\begin{cor}\label{cor_hE_iso2}
All differentials~$d^r$ with $r\ge 1$ of the spectral sequence~$\hE^*_{*,*}(\fN)$ either from or to the groups $\hE^r_{n,\,3g-5-2n}(\fN)$ are trivial. Hence
$$
\hE^{\infty}_{n,\,3g-5-2n}(\fN)=\hE^1_{n,\,3g-5-2n}(\fN)\cong\HH_{3g-5-2n}(H_N).
$$
\end{cor}

\section{Proof of Theorem~\ref{theorem_Abelian_explicit}}\label{section_Abelian_proof}

Recall that we have a primitive homology class $x\in \HH_1(S_g)$ fixed throughout the whole paper. This homology class was used in the construction of the abelian cycles~$A_{\CU,\bb}$ in Subsection~\ref{subsection_A}. We consider the complex of cycles~$\B_g=\B_g(x)$ for the same homology class~$x$, the spectral sequence~$E^*_{*,*}$ from Section~\ref{section_CL} for the action of~$\I_g$ on~$\B_g$, and the corresponding filtration~$\CF_{*,*}$ in $\HH_*(\I_g)$.

Each abelian cycle
\begin{equation*}
A_{\CU,\bb}=\CA\bigl(
T_{\beta_1}T_{\beta_1'}^{-1},\ldots,T_{\beta_{g-2}}T_{\beta_{g-2}'}^{-1}, T_{\delta_1},\ldots, T_{\delta_{g-1}}
\bigr)
\end{equation*}
corresponds to a $\beta\delta$-multicurve
\begin{equation} \label{eq_Gamma}
\Gamma=B\cup\Delta=\beta_1\cup\beta_1'\cup\cdots\cup\beta_{g-2}\cup\beta_{g-2}'\cup \delta_1\cup\cdots\cup\delta_{g-1}
\end{equation}
such that its $\delta$-part $\Delta=\delta_1\cup\cdots\cup\delta_{g-1}$ provides the splitting~$\CU$ in homology and  $$\bigl([\beta_1],\ldots,[\beta_{g-2}]\bigr)=\bb.$$
To detect the classes~$A_{\CU,\bb}$ in the spectral sequence~$E^*_{*,*}$ we will use type one $\alpha$-multicurves, see Definition~\ref{defin_alpha}.

\begin{definition}\label{defin_compat_alpha_gamma}
 A $\beta\delta$-multicurve
 $$
 \Gamma=B\cup\Delta=\beta_1\cup\beta_1'\cup\cdots\cup\beta_{g-2}\cup\beta_{g-2}'\cup\delta_1\cup\cdots\cup\delta_{g-1}
 $$
 and a type one $\alpha$-multicurve
 $$
 N=\alpha_0\cup\alpha_1\cup\alpha_1'\cup\cdots\cup\alpha_{g-2}\cup\alpha_{g-2}'\cup\alpha_{g-1}
 $$
 are said to be \textit{compatible} with each other if they are arranged as shown in Fig.~\ref{fig_alpha_beta_delta}, that is,
 \begin{enumerate}
  \item $N$ is disjoint from~$\Delta$,
  \item the curves $\alpha_0$ and~$\alpha_{g-1}$ lie in the once-punctured tori~$X_0$ and~$X_{g-1}$ bounded by~$\delta_1$ and~$\delta_{g-1}$, respectively,
  \item for $1\le i\le g-2$,
  \begin{itemize}
  \item the curves $\alpha_i$ and~$\alpha_i'$ lie in the twice-punctured torus~$X_i$ bounded by~$\delta_i$ and~$\delta_{i+1}$,
  \item the geometric intersection number of each of the four pairs of simple closed curves $(\alpha_i,\beta_i)$, $(\alpha_i,\beta_i')$, $(\alpha_i',\beta_i)$, and~$(\alpha_i',\beta_i')$ is equal to~$1$.
  \end{itemize}
 \end{enumerate}
\end{definition}

\begin{figure}
\begin{tikzpicture}[scale=.8]
\footnotesize

\draw[mygreen, thick] (1.5,-1.5) arc (270:90:0.3 and 1.5)
node[pos=0,below]{$\delta_1$};
\draw[mygreen, thick, dashed] (1.5,-1.5) arc (-90:90:0.3 and 1.5);

\draw[mygreen, thick] (4.5,-1.5) arc (270:90:0.3 and 1.5)
node[pos=0,below]{$\delta_2$};
\draw[mygreen, thick, dashed] (4.5,-1.5) arc (-90:90:0.3 and 1.5);

\draw[mygreen, thick] (7.5,-1.5) arc (270:90:0.3 and 1.5)
node[pos=0,below]{$\delta_3$};
\draw[mygreen, thick, dashed] (7.5,-1.5) arc (-90:90:0.3 and 1.5);

\draw[mygreen, thick] (10.3,-1.5) arc (270:90:0.3 and 1.5)
node[pos=0,below]{$\delta_{g-2}$};
\draw[mygreen, thick, dashed] (10.3,-1.5) arc (-90:90:0.3 and 1.5);

\draw[mygreen, thick] (13.7,-1.5) arc (270:90:0.3 and 1.5)
node[pos=0,below]{$\delta_{g-1}$};
\draw[mygreen, thick, dashed] (13.7,-1.5) arc (-90:90:0.3 and 1.5);

\draw[mygreen, thick, snake=brace] (-1.4,1.75) -- (1.4,1.75)
node[pos=.5,above]{$X_0$};
\draw[mygreen, thick, snake=brace] (1.6,1.75) -- (4.4,1.75)
node[pos=.5,above]{$X_1$};
\draw[mygreen, thick, snake=brace] (4.6,1.75) -- (7.4,1.75)
node[pos=.5,above]{$X_2$};
\draw[mygreen, thick, snake=brace] (10.4,1.75) -- (13.6,1.75)
node[pos=.5,above]{$X_{g-2}$};
\draw[mygreen, thick, snake=brace] (13.8,1.75) -- (16.4,1.75)
node[pos=.5,above]{$X_{g-1}$};

\draw[myblue, thick] (3,1.5) arc (90:270:.15 and .55); 
\draw[myblue, thick,->] (3,1.5) arc (90:185:.15 and .55); 
\draw[myblue, thick, dashed] (3,1.5) arc (90:-90:.15 and .55)
node[pos=0.5,right]{${}\!\beta_{1}$}; 

\draw[myblue, thick] (3,-.4) arc (90:270:.15 and .55); 
\draw[myblue, thick,->] (3,-1.5) arc (270:175:.15 and .55); 
\draw[myblue, thick, dashed] (3,-.4) arc (90:-90:.15 and .55)
node[pos=0.45,right]{${}\!\beta_{1}'$}; 

\draw[myblue, thick] (6,1.5) arc (90:270:.15 and .55); 
\draw[myblue, thick,->] (6,1.5) arc (90:185:.15 and .55); 
\draw[myblue, thick, dashed] (6,1.5) arc (90:-90:.15 and .55)
node[pos=0.5,right]{${}\!\beta_{2}$}; 

\draw[myblue, thick] (6,-.4) arc (90:270:.15 and .55); 
\draw[myblue, thick,->] (6,-1.5) arc (270:175:.15 and .55); 
\draw[myblue, thick, dashed] (6,-.4) arc (90:-90:.15 and .55)
node[pos=0.45,right]{${}\!\beta_{2}'$}; 

\draw[myblue, thick] (12,1.5) arc (90:270:.15 and .55)
node[pos=0.15,left=8pt,scale=.85,rotate=90]{\tiny${}\!\beta_{g-2}$}; 
\draw[myblue, thick,->] (12,1.5) arc (90:185:.15 and .55); 
\draw[myblue, thick, dashed] (12,1.5) arc (90:-90:.15 and .55); 

\draw[myblue, thick] (12,-.4) arc (90:270:.15 and .55)
node[pos=0.1,left=9pt,scale=.85,rotate=90]{\tiny${}\!\beta_{g-2}'$}; 
\draw[myblue, thick,->] (12,-1.5) arc (270:170:.15 and .55); 
\draw[myblue, thick, dashed] (12,-.4) arc (90:-90:.15 and .55); 

\draw[red, thick] (-1.5,0) arc (180:360:.55 and .15)
node[pos=0.5,below]{${}\!\alpha_0$}; 
\draw[red, thick,->] (-.4,0) arc (0:-95:.55 and .15); 
\draw[red, thick, dashed] (-.4,0) arc (0:180:.55 and .15); 

\draw[red, thick] (15.4,0) arc (180:360:.55 and .15)
node[pos=0.5,below]{${}\!\alpha_{g-1}$}; 
\draw[red, thick,->] (15.4,0) arc (180:275:.55 and .15); 
\draw[red, thick, dashed] (16.5,0) arc (0:180:.55 and .15); 

\draw[red, thick] (3,0) circle (.6);
\draw[red, thick,->] (3.6,0) arc (0:185:.6) node[pos=.75,above=1pt]{$\alpha_1$};
\draw[red, thick] (3,1.5) arc (90:270:1 and 1.5)
node[pos=.2,left]{$\alpha_1'$};
\draw[red, thick,->] (3,1.5) arc (90:182:1 and 1.5);
\draw[red, thick,dashed] (3,1.5) arc (90:-90:1 and 1.5);

\draw[red, thick] (6,0) circle (.6);
\draw[red, thick,->] (6.6,0) arc (0:185:.6) node[pos=.75,above=1pt]{$\alpha_2$};
\draw[red, thick] (6,1.5) arc (90:270:1 and 1.5)
node[pos=.2,left]{$\alpha_2'$};
\draw[red, thick,->] (6,1.5) arc (90:182:1 and 1.5);
\draw[red, thick,dashed] (6,1.5) arc (90:-90:1 and 1.5);

\draw[red, thick] (12,0) circle (.6);
\draw[red, thick,->] (12.6,0) arc (0:185:.6) ;
\draw[red, thick] (12.6,0) arc (0:180:.6) node[sloped,pos=0.25,above=-2pt]{\tiny$\alpha_{g-2}$};
\draw[red, thick] (12,1.5) arc (90:270:1.1 and 1.5);
\draw[red, thick,->] (12,1.5) arc (90:182:1.1 and 1.5);
\draw[red, thick,dashed] (12,1.5) arc (90:-90:1.1 and 1.5)
node[sloped,pos=0.29,above=-2pt]{\tiny$\alpha_{g-2}'$};

\draw[black, ultra thick] (0,1.5) -- (15,1.5) arc (90:-90:1.5) -- (0,-1.5) arc (270:90:1.5);
\draw[black, ultra thick] (0,0) circle (.4);
\draw[black, ultra thick] (3,0) circle (.4);
\draw[black, ultra thick] (6,0) circle (.4);
\draw[black, ultra thick] (12,0) circle (.4);
\draw[black, ultra thick] (15,0) circle (.4);

\fill[black] (9,0) circle (1.5pt);
\fill[black] (9.3,0) circle (1.5pt);
\fill[black] (8.7,0) circle (1.5pt);

\end{tikzpicture}
\caption{A $\beta\delta$-multicurve and an $\alpha$-multicurve compatible with each other}\label{fig_alpha_beta_delta}
\end{figure}

\begin{remark}
There is exactly one way to re-order the components of a type one $\alpha$-multicurve~$N$ so that conditions~(1)--(4) in Definition~\ref{defin_alpha} will still be satisfied. Namely, we can put
\begin{equation}\label{eq_reorder}
\begin{aligned}
\overline{\alpha}_0&=\alpha_{g-1}\,,&
\overline{\alpha}_{g-1}&=\alpha_{0}\,,&
\overline{\alpha}_i&=\alpha_{g-1-i}'\,,&
\overline{\alpha}_i^{\,\prime}&=\alpha_{g-1-i}\,,&& i=1,\ldots,g-2.
\end{aligned}
\end{equation}
Then conditions~(1)--(4) will be satisfied with $Y_i$ replaced with~$\overline{Y}_{\!\!i}=Y_{g-i}$ and $l_i$ replaced with $\bar{l}_i=l_{g-i-1}$. We will consider~$N$ as a multicurve  without chosen order of its components so that the two representations of~$N$ in form~\eqref{eq_N} are equally matched. In particular, this means that we need to refine Definition~\ref{defin_compat_alpha_gamma} by saying that conditions~\mbox{(1)--(3)} in this definition either must be satisfied or must be satisfied after re-ordering the components of~$N$ by~\eqref{eq_reorder}.
\end{remark}

\begin{propos}\label{propos_compatible_exist}
 For each $\beta\delta$-multicurve~$\Gamma$ of the form~\eqref{eq_Gamma} there exists a type one $\alpha$-multicurve compatible with it.
\end{propos}

\begin{proof}
Let $\CU=(U_0,\ldots,U_{g-1})$ be the splitting given by the $\delta$-part of~$\Gamma$ and let $$
x=l_0a_0+\ldots+l_{g-1}a_{g-1}
$$
be the corresponding decomposition of~$x$ with primitive $a_i\in U_i$ and $l_i>0$.  (Recall that the strict positivity of~$l_i$ is part of the definition of a $\beta\delta$-multicurve, see Subsection~\ref{subsection_A}.)
Choose oriented simple closed curves~$\alpha_0\subset X_0$ and~$\alpha_{g-1}\subset X_{g-1}$ in the primitive homology classes~$a_0$ and~$a_{g-1}$, respectively. From the definition of a $\beta\delta$-multicurve it follows that $a_i\cdot b_i=1$ for $i=1,\ldots,g-2$, where $b_i=[\beta_i]$. So, for every $i=1,\ldots,g-2$, we can choose an oriented simple closed curve~$\alpha_i\subset X_i$ in the homology class~$a_i$ such that $\alpha_i$ has a unique transverse intersection point with each of the curves~$\beta_i$ and~$\beta_i'$.  Put $\alpha_i'=T_{\beta_i}T_{\beta_i'}^{-1}(\alpha_i)$. Then the curves~$\alpha_i$, $\alpha_i'$, $\beta_i$, and~$\delta_i$ are arranged as shown in Fig.~\ref{fig_alpha_beta_delta}, so
$$
N=\alpha_0\cup\alpha_1\cup\alpha_1'\cup\cdots\cup\alpha_{g-2}\cup\alpha_{g-2}'\cup\alpha_{g-1}
$$
is a type one $\alpha$-multicurve compatible with~$\Gamma$.
\end{proof}

\begin{propos}\label{propos_Phi}
Each abelian cycle~$A_{\CU,\bb}$ with $\CU\in\CS$ and $\bb\in\fB(\CU)$ lies in the subgroup $\CF_{g-2,\,g-1}\subset \HH_{2g-3}(\I_g)$. Moreover, suppose that $A_{\CU,\bb}$ corresponds to a $\beta\delta$-multicurve~$\Gamma$ of the form~\eqref{eq_Gamma} and $N$ is a type one $\alpha$-multicurve compatible with~$\Gamma$. Then the image of~$A_{\CU,\bb}$ under the projection
$$
\CF_{g-2,\,g-1}\to\CF_{g-2,\,g-1}/\CF_{g-3,\,g}=E^{\infty}_{g-2,\,g-1}
$$
is represented by the element
$$
\iota_{P_N}\bigl(\CA(T_{\delta_1},\ldots,T_{\delta_{g-1}})\bigr)\in E^1_{g-2,\,g-1}.
$$
\end{propos}

\begin{proof}
  We put $f_i=T_{\beta_i}T_{\beta_i'}^{-1}$ for $i=1,\ldots,g-2$. Let $N_i$ be the $(2g-3)$-component oriented multicurve obtained from~$N$ by removing the component~$\alpha_i'$ and  $N_i'$ be the $(2g-3)$-component oriented multicurve obtained from~$N$ by removing the component~$\alpha_i$. Let $\varphi\colon[0,1]^{g-2}\to P_N$ be the isomorphism given by
 \begin{equation*}
\varphi(t_1,\ldots,t_{g-2})= l_0\alpha_0+l_{g-1}\alpha_{g-1}+\sum_{i=1}^{g-2}l_i\bigl((1-t_i)\alpha_i+t_i\alpha_i'\bigr)
\end{equation*}
 Then $P_{N_i}=\varphi(F_i)$ and~$P_{N_i'}=\varphi(F_i')$, where $F_i$ and~$F_i'$ are the facets of the cube~$[0,1]^{g-2}$ given by~$t_i=0$ and~$t_i=1$, respectively. Since $f_i(\alpha_i)=\alpha_i'$ and~$f_i$ stabilizes~$\alpha_j$ and~$\alpha_j'$ unless $j=i$, we have $f_i(N_i)=N_i'$. Moreover, the action by~$f_i$ yields the homeomorphism~$P_{N_i}\to P_{N_i'}$ that is identified by~$\varphi$ with the parallel translation taking~$F_i$ to~$F_i'$. So the assertion of the proposition follows immediately from Lemma~\ref{propos_spectral} for  the cell~$P_N$,  the homology class
$$
u=\CA(T_{\delta_1},\ldots,T_{\delta_{g-1}})\in \HH_{g-1}\bigl(\Stab_{\I_g}(N)\bigr),
$$
the homeomorphism~$\varphi$, and the mapping classes $f_1,\ldots,f_{g-2}$.
\end{proof}

Suppose that $N$ is an arbitrary type one $\alpha$-multicurve and let $\fN$ denote the $\I_g$-orbit of~$N$.
Combining Facts~\ref{fact_E1} and~\ref{fact_Einfty}, the construction of the auxiliary spectral sequence~$\hE^*_{*,*}$ from Section~\ref{section_aux_spectral}, and Corollary~\ref{cor_hE_iso}, we obtain a commutative diagram
\begin{equation}\label{eq_comm_diag}
\begin{tikzcd}[bezier bounding box = true]
 & \HH_{g-1}\bigl(\Stab_{\I_g}(N)\bigr) \arrow[r,"(j_N)_*" above]
 \arrow[d, "\iota_{P_N}"]
 & \HH_{g-1}(H_N) \arrow[d,"\cong" left, "\iota_{P_N}" right]\\
  & E^1_{g-2,\,g-1} \arrow[r, "\hPi_{\fN}"] \arrow[d, dashed, two heads]  & \hE^1_{g-2,\,g-1}(\fN) \arrow[d, equal]\\
  \CF_{g-2,\,g-1} \arrow[r, twoheadrightarrow] \arrow[rruu, controls = {+(13,-2.3) and +(3,-.5)}, "{}\quad \Phi_N" below] & E^{\infty}_{g-2,\,g-1} \arrow[r, "\hPi_{\fN}"] & \hE^{\infty}_{g-2,\,g-1}(\fN)
 \end{tikzcd}
\end{equation}
Here the dashed arrow is the canonical partial surjective homomorphism. Its domain coincides with the intersection of kernels of all differentials of the spectral sequence. We denote by~$\Phi_N$ the homomorphism obtained by passing around this diagram in the counterclockwise direction from the lower left  to the upper right corner.

\begin{propos}\label{propos_Phi_formula}
 Suppose that an abelian cycle~$A_{\CU,\bb}$ corresponds to a $\beta\delta$-multicurve~$\Gamma$ of the form~\eqref{eq_Gamma} and $N$ is a type one $\alpha$-multicurve compatible with~$\Gamma$. Then
 \begin{equation}\label{eq_PhiNform}
 \Phi_N(A_{\CU,\bb})=A_{\CU}=\CA(T_{\delta_1},\ldots,T_{\delta_{g-1}}).
 \end{equation}
 Besides, for any type one $\alpha$-multicurve~$\widetilde{N}$ that does not belong to the $\I_g$-orbit of~$N$, we have
 \begin{equation}\label{eq_PhiNform_zero}
 \Phi_{\widetilde{N}}(A_{\CU,\bb})=0.
 \end{equation}
\end{propos}

\begin{proof}
 Formula~\eqref{eq_PhiNform} follows immediately from Proposition~\ref{propos_Phi} and the commutativity of the diagram~\eqref{eq_comm_diag}. Now, let $\widetilde{\fN}$ be the $\I_g$-orbit of~$\widetilde{N}$. Consider the commutative diagram~\eqref{eq_comm_diag} with $N$ and~$\fN$ replaced by~$\widetilde{N}$ and~$\widetilde{\fN}$, respectively. Since the multicurves~$N$ and~$\widetilde{N}$ consist of the same number of components and lie in different $\I_g$-orbits, we see that $N$ does not contain a submulticurve that belongs to~$\widetilde{\mathfrak{N}}$. Then from the construction of the homomorphism
 $$
 \hPi_{\widetilde{\fN}}\colon E^1_{g-2,\,g-1}\to\hE^1_{g-2,\,g-1}\bigl(\widetilde{\fN}\bigr)
 $$
 (see Section~\ref{section_aux_spectral}) it follows that $\hPi_{\widetilde{\fN}}\circ\iota_{P_N}=0$. So formula~\eqref{eq_PhiNform_zero} also follows from Proposition~\ref{propos_Phi}.
\end{proof}

\begin{remark}
 By Corollary~\ref{cor_Hg-1inv} we have $A_{\CU}\ne 0$. So Proposition~\ref{propos_Phi_formula} implies that all type one $\alpha$-multicurves compatible with the same $\beta\delta$-multicurve~$\Gamma$ lie in the same $\I_g$-orbit.
\end{remark}

We are now ready to prove Theorem~\ref{theorem_Abelian_explicit}.

\begin{proof}[Proof of Theorem~\ref{theorem_Abelian_explicit}]
The set of all type one $\alpha$-multicurves is decomposed into $\I_g$-orbits. Let~$\mathcal{N}$ be a set of representatives for this decomposition, i.\,e., a set containing exactly one $\alpha$-multicurve in each $\I_g$-orbit. For each abelian cycle~$A_{\CU,\bb_{\CU}}$ consider a corresponding $\beta\delta$-multicurve~$\Gamma_{\CU}$ and a type one $\alpha$-multicurve~$N_{\CU}$ compatible with~$\Gamma_{\CU}$. Since $\I_g$ acts trivially on the homology of itself, we may replace~$\Gamma_{\CU}$ and~$N_{\CU}$ with~$h(\Gamma_{\CU})$ and~$h(N_{\CU})$, respectively, where $h\in\I_g$, and the corresponding abelian cycle will not change. So we may achieve that $N_{\CU}\in\mathcal{N}$ for all~$\CU$. Then the $\alpha$-multicurves~$N_{\CU_1}$ and~$N_{\CU_2}$ corresponding to any two different splittings either coincide or lie in different $\I_g$-orbits. Consider the homomorphism
$$
\Phi=\bigoplus_{N\in\mathcal{N}}\Phi_N\colon \CF_{g-2,\,g-1}\to\bigoplus_{N\in\mathcal{N}}\HH_{g-1}(H_N).
$$
By Proposition~\ref{propos_Phi_formula}, this homomorphism takes each abelian cycle~$A_{\CU,\bb_{\CU}}$ to the homology class~$A_{\CU}$ in the summand~$\HH_{g-1}(H_{N_{\CU}})$. By Corollary~\ref{cor_Hg-1inv} such homology classes for different splittings~$\CU$ are linearly independent. So the elements $\Phi(A_{\CU,\bb_{\CU}})$ and hence the elements~$A_{\CU,\bb_{\CU}}$ are linearly independent.
\end{proof}

\section{Proof of Theorem~\ref{theorem_explicit}}\label{section_main_proof}

We will follow the same line as in the proof of Theorem~\ref{theorem_Abelian_explicit}.
Each homology class
\begin{equation*}
A_{\Gamma,\psi,\xi}=\CA\bigl(
T_{\beta_1}T_{\beta_1'}^{-1},\ldots,T_{\beta_{n}}T_{\beta_{n}'}^{-1},
T_{\delta_1},\ldots, T_{\delta_{n}};\psi_*(\xi)
\bigr)\in \HH_{3g-5-n}(\I_g).
\end{equation*}
corresponds to a triple $(\Gamma,\psi,\xi)$ such that
\begin{itemize}
\item $\Gamma$ is a $\beta\delta$-multicurve
\begin{equation} \label{eq_Gamma2}
\Gamma=B\cup\Delta= \beta_1\cup\beta_1'\cup\cdots\cup\beta_n\cup\beta_n'\cup\delta_1\cup\cdots\cup\delta_n
\end{equation}
for some splitting~$\CU=(U_0,\ldots,U_n)$ belonging to~$\CS_n$,
\item $\psi\colon S_{g'}^1\hookrightarrow S_g\setminus \Gamma$ is an embedding,
\item $\xi\in \HH_{3g'-2}\bigl(\I_{g'}^1\bigr)$.
\end{itemize}
Hereinafter, we put $g'=g-n-1$.

To detect the classes~$A_{\Gamma,\psi,\xi}$ in the spectral sequence~$E^*_{*,*}$ we will use $(2n+1)$-component type two $\alpha$-multicurves, see Definition~\ref{defin_alpha2}.

\begin{definition}\label{defin_compat_alpha_gamma2}
 A pair $(\Gamma,\psi)$ and a $(2n+1)$-component (type two) $\alpha$-multicurve
 $$
 N=\alpha_0\cup\alpha_1\cup\alpha_1'\cup\cdots\cup\alpha_{n}\cup\alpha_{n}'
 $$
 are said to be \textit{compatible} with each other if $\Gamma$, $\psi\bigl(S_{g'}^1\bigr)$, and~$N$ are arranged as shown in Fig.~\ref{fig_alpha_beta_delta2}, that is,
 \begin{enumerate}
  \item $N$ is disjoint from~$\Delta$,
  \item the curve $\alpha_0$ lies in the once-punctured torus~$X_0$ bounded by~$\delta_1$,
  \item for $1\le i\le n$,
  \begin{itemize}
  \item the curves $\alpha_i$ and~$\alpha_i'$ lie in the twice-punctured torus~$X_i$ bounded by~$\delta_i$ and~$\delta_{i+1}$ if $i<n$ and in the twice-punctured torus~$Z$ bounded by~$\delta_n$ and~$\varepsilon=\psi\bigl(\partial S_{g'}^1\bigr)$ if $i=n$,
  \item the geometric intersection number of each of the four pairs of simple closed curves $(\alpha_i,\beta_i)$, $(\alpha_i,\beta_i')$, $(\alpha_i',\beta_i)$, and~$(\alpha_i',\beta_i')$ is equal to~$1$.
  \end{itemize}
 \end{enumerate}
\end{definition}

\begin{figure}
\begin{tikzpicture}[scale=.8]
\footnotesize

\draw[mygreen, thick] (1.5,-1.5) arc (270:90:0.3 and 1.5)
node[pos=0,below]{$\delta_1$};
\draw[mygreen, thick, dashed] (1.5,-1.5) arc (-90:90:0.3 and 1.5);

\draw[mygreen, thick] (4.5,-1.5) arc (270:90:0.3 and 1.5)
node[pos=0,below]{$\delta_2$};
\draw[mygreen, thick, dashed] (4.5,-1.5) arc (-90:90:0.3 and 1.5);

\draw[mygreen, thick] (7.5,-1.5) arc (270:90:0.3 and 1.5)
node[pos=0,below]{$\delta_n$};
\draw[mygreen, thick, dashed] (7.5,-1.5) arc (-90:90:0.3 and 1.5);

\fill[violet!7] (10.5,-1.5) arc (270:90:0.3 and 1.5) -- (15,1.5) arc (90:-90:1.5) -- cycle
(12,0) circle (.4) (15,0) circle (.4);
\fill[violet!15] (10.5,-1.5) arc (-90:90:0.3 and 1.5) -- (15,1.5) arc (90:-90:1.5) -- cycle
(12,0) circle (.4) (15,0) circle (.4);

\draw[violet, very thick] (10.5,-1.5) arc (270:90:0.3 and 1.5)
node[pos=0,below]{\small$\varepsilon$};
\draw[violet, very thick,dashed] (10.5,-1.5) arc (-90:90:0.3 and 1.5);

\draw[mygreen, thick, snake=brace] (-1.4,1.75) -- (1.4,1.75)
node[pos=.5,above]{$X_0$};
\draw[mygreen, thick, snake=brace] (1.6,1.75) -- (4.4,1.75)
node[pos=.5,above]{$X_1$};
\draw[violet, thick, snake=brace] (7.6,1.75) -- (10.4,1.75)
node[pos=.5,above]{$Z$};
\draw[violet, thick, snake=brace] (10.6,1.75) -- (16.4,1.75)
node[pos=.5,above]{$\psi\bigl(S_{g'}^1\bigr)$};
\draw[mygreen, thick, snake=brace] (7.6,2.75) -- (16.4,2.75)
node[pos=.5,above]{$X_n$};

\draw[myblue, thick] (3,1.5) arc (90:270:.15 and .55); 
\draw[myblue, thick,->] (3,1.5) arc (90:185:.15 and .55); 
\draw[myblue, thick, dashed] (3,1.5) arc (90:-90:.15 and .55)
node[pos=0.5,right]{${}\!\beta_{1}$}; 

\draw[myblue, thick] (3,-.4) arc (90:270:.15 and .55); 
\draw[myblue, thick,->] (3,-1.5) arc (270:175:.15 and .55); 
\draw[myblue, thick, dashed] (3,-.4) arc (90:-90:.15 and .55)
node[pos=0.45,right]{${}\!\beta_{1}'$}; 

\draw[myblue, thick] (9,1.5) arc (90:270:.15 and .55); 
\draw[myblue, thick,->] (9,1.5) arc (90:185:.15 and .55); 
\draw[myblue, thick, dashed] (9,1.5) arc (90:-90:.15 and .55)
node[pos=0.5,right]{${}\!\beta_{n}$}; 

\draw[myblue, thick] (9,-.4) arc (90:270:.15 and .55); 
\draw[myblue, thick,->] (9,-1.5) arc (270:175:.15 and .55); 
\draw[myblue, thick, dashed] (9,-.4) arc (90:-90:.15 and .55)
node[pos=0.44,right]{${}\!\beta_{n}'$};

\draw[red, thick] (-1.5,0) arc (180:360:.55 and .15)
node[pos=0.5,below]{${}\!\alpha_0$}; 
\draw[red, thick,->] (-.4,0) arc (0:-95:.55 and .15); 
\draw[red, thick, dashed] (-.4,0) arc (0:180:.55 and .15);

\draw[red, thick] (3,0) circle (.6);
\draw[red, thick,->] (3.6,0) arc (0:185:.6) node[pos=.75,above=1pt]{$\alpha_1$};
\draw[red, thick] (3,1.5) arc (90:270:1 and 1.5)
node[pos=.2,left]{$\alpha_1'$};
\draw[red, thick,->] (3,1.5) arc (90:182:1 and 1.5);
\draw[red, thick,dashed] (3,1.5) arc (90:-90:1 and 1.5);

\draw[red, thick] (9,0) circle (.6);
\draw[red, thick,->] (9.6,0) arc (0:185:.6) node[pos=.75,above=1pt]{{}\!$\alpha_n$};
\draw[red, thick] (9,1.5) arc (90:270:1 and 1.5)
node[pos=.2,left]{$\alpha_n'$};
\draw[red, thick,->] (9,1.5) arc (90:182:1 and 1.5);
\draw[red, thick,dashed] (9,1.5) arc (90:-90:1 and 1.5);

\draw[black, ultra thick] (0,1.5) -- (15,1.5) arc (90:-90:1.5) -- (0,-1.5) arc (270:90:1.5);
\draw[black, ultra thick] (0,0) circle (.4);
\draw[black, ultra thick] (3,0) circle (.4);
\draw[black, ultra thick] (9,0) circle (.4);
\draw[black, ultra thick] (12,0) circle (.4);
\draw[black, ultra thick] (15,0) circle (.4);

\fill[black] (6,0) circle (1.5pt);
\fill[black] (6.3,0) circle (1.5pt);
\fill[black] (5.7,0) circle (1.5pt);
\fill[black] (13.2,0) circle (1.5pt);
\fill[black] (13.5,0) circle (1.5pt);
\fill[black] (13.8,0) circle (1.5pt);

\end{tikzpicture}
\caption{A pair~$(\Gamma,\psi)$ and an $\alpha$-multicurve~$N$ compatible with each other}\label{fig_alpha_beta_delta2}
\end{figure}

The proofs of the following two propositions repeat literally the proofs of Propositions~\ref{propos_compatible_exist} and~\ref{propos_Phi}, respectively.

\begin{propos}\label{propos_compatible_exist2}
 For each $\beta\delta$-multicurve~$\Gamma$ of the form~\eqref{eq_Gamma2} and each embedding $\psi\colon S_{g'}^1\hookrightarrow S_g\setminus\Gamma$, there exists a $(2n+1)$-component $\alpha$-multicurve compatible with the pair~$(\Gamma,\psi)$.
\end{propos}

\begin{propos}\label{propos_Phi2}
Each homology class~$A_{\Gamma,\psi,\xi}$ with $(\Gamma,\psi,\xi)$ as above lies in the subgroup $\CF_{n,\,3g-5-2n}\subset \HH_{3g-5-n}(\I_g)$. Moreover, suppose that $N$ is a $(2n+1)$-component $\alpha$-multicurve compatible with the pair~$(\Gamma,\psi)$. Then the image of~$A_{\Gamma,\psi,\xi}$ under the projection
$$
\CF_{n,\,3g-5-2n}\to\CF_{n,\,3g-5-2n}/\CF_{n-1,\,3g-4-2n}=E^{\infty}_{n,\,3g-5-2n}
$$
is represented by the element
$$
\iota_{P_N}\bigl(\CA\bigl(T_{\delta_1},\ldots,T_{\delta_{n}};\psi_*(\xi)\bigr)\bigr)\in E^1_{n,\,3g-5-2n}.
$$
\end{propos}

Let $N$ be an arbitrary $(2n+1)$-component $\alpha$-multicurve and $\fN$ be the $\I_g$-orbit of~$N$. Combining Facts~\ref{fact_E1} and~\ref{fact_Einfty}, the construction of the auxiliary spectral sequence~$\hE^*_{*,*}$ from Section~\ref{section_aux_spectral}, and Corollary~\ref{cor_hE_iso2}, we obtain the following analog of the commutative diagram~\eqref{eq_comm_diag}:
\begin{equation*}\label{eq_comm_diag2}
\begin{tikzcd}[bezier bounding box = true]
 & \HH_{3g-5-2n}\bigl(\Stab_{\I_g}(N)\bigr) \arrow[r,"(j_N)_*" above]
 \arrow[d, "\iota_{P_N}"]
 & \HH_{3g-5-2n}(H_N) \arrow[d,"\cong" left, "\iota_{P_N}" right]\\
  & E^1_{n,\,3g-5-2n} \arrow[r, "\hPi_{\fN}"] \arrow[d, dashed, two heads]  & \hE^1_{n,\,3g-5-2n}(\fN) \arrow[d, equal]\\
  \CF_{n,\,3g-5-2n} \arrow[r, twoheadrightarrow] \arrow[rruu, controls = {+(15,-2.5) and +(3,-.2)}, "{}\quad \Phi_N" below] & E^{\infty}_{n,\,3g-5-2n} \arrow[r, "\hPi_{\fN}"] & \hE^{\infty}_{n,\,3g-5-2n}(\fN)
 \end{tikzcd}
\end{equation*}
\smallskip

Recall that according to Proposition~\ref{propos_H_homology_total} we have
$$
\HH_{3g-5-2n}(H_N)=\HH_n(F_N)\otimes \HH_{3g-5-3n}(G_N),
$$
where $F_N$ and~$G_N$ are the subgroups defined in Subsection~\ref{subsection_type2}.
The following proposition is proved likewise Proposition~\ref{propos_Phi_formula}.

\begin{propos}\label{propos_Phi_formula2}
 Consider a homology class~$A_{\Gamma,\psi,\xi}$ as above. Let $\CU\in\CS_n$ be the splitting given by the $\delta$-part of~$\Gamma$. Let $N$ be a $(2n+1)$-component $\alpha$-multicurve compatible with the pair~$(\Gamma,\psi)$. Then
 \begin{equation*}
 \Phi_N(A_{\Gamma,\psi,\xi})=A_{\CU}\times \psi_*(\xi)\in \HH_{3g-5-2n}(H_N).
 \end{equation*}
 Besides, for any $(2n+1)$-component $\alpha$-multicurve~$\widetilde{N}$ that does not belong to the $\I_g$-orbit of~$N$, we have
 \begin{equation*}
 \Phi_{\widetilde{N}}(A_{\Gamma,\psi,\xi})=0.
 \end{equation*}
\end{propos}

We are now ready to prove Theorem~\ref{theorem_explicit}.

\begin{proof}[Proof of Theorem~\ref{theorem_explicit}] As in the proof of Theorem~\ref{theorem_Abelian_explicit}, choose a set~$\mathcal{N}_n$ of representatives of $\I_g$-orbits of $(2n+1)$-component $\alpha$-multicurves. For each pair~$(\Gamma_{\CU},\psi_{\CU})$, consider a $(2n+1)$-component $\alpha$-multicurve~$N_{\CU}$ compatible with it. Again, as in the proof of Theorem~\ref{theorem_Abelian_explicit}, we may assume that $N_{\CU}\in\mathcal{N}_n$ for all~$\CU$. Consider the homomorphism
$$
\Phi=\bigoplus_{N\in\mathcal{N}_n}\Phi_N\colon \CF_{n,\,3g-5-2n}\to\bigoplus_{N\in\mathcal{N}_n}\HH_{3g-5-2n}(H_N).
$$
By Proposition~\ref{propos_Phi_formula2} this homomorphism takes each homology class~$A_{\Gamma_{\CU},\psi_{\CU},\xi_{\lambda}}$ to the element~$A_{\CU}\times (\psi_{\CU})_*(\xi_{\lambda})$ in the summand
$$
\HH_{3g-5-2n}(H_N)=\HH_n(F_N)\otimes \HH_{3g-5-3n}(G_N)=\HH_n(F_N)\otimes \HH_{3g'-2}(G_N),
$$
The isomorphism from Proposition~\ref{propos_iso_G} identifies the homomorphism
$$
(\psi_{\CU})_*\colon \HH_{3g'-2}\bigl(\I_{g'}^1\bigr) \to \HH_{3g'-2}(G_N)
$$
with the homomorphism
$$
\HH_{3g'-2}\bigl(\I_{g'}^1\bigr) \to \HH_{3g'-2}\bigl(\Stab_{\I_{g'+1}}(v)\bigr)
$$
from~\eqref{eq_mono_sequence}. Therefore, $(\psi_{\CU})_*$ is injective. Consequently, for each~$\CU$, the classes~$(\psi_{\CU})_*(\xi_{\lambda})$ with $\lambda\in\Lambda$ are linearly independent. So it follows from Proposition~\ref{propos_H_homology_total} that the elements $\Phi(A_{\Gamma_{\CU},\psi_{\CU},\xi_{\lambda}})$ and hence the elements~$A_{\Gamma_{\CU},\psi_{\CU},\xi_{\lambda}}$ are linearly independent.
\end{proof}

\end{document}